\documentclass[11pt,reqno]{amsart}
\usepackage[margin=1in]{geometry}
\usepackage{adjustbox,enumerate}
\usepackage{amsmath,amsfonts,amssymb,amsthm}
\usepackage{mathtools}
\usepackage{commath}

\usepackage{mlmodern}

\usepackage{graphicx} %package to manage images
\graphicspath{ {images/} }
\usepackage[rightcaption]{sidecap}
\usepackage{caption}
\usepackage{subcaption}
\usepackage{epsfig}

\usepackage{siunitx}
\usepackage{hyperref,xcolor}
\hypersetup{
    colorlinks,
    linkcolor={red!50!black},
    citecolor={blue!50!black},
    urlcolor={blue!80!black}
}

\let\Re\undefine
\let\Im\undefine
\DeclareMathOperator{\Re}{Re}
\DeclareMathOperator{\Im}{Im}

\DeclareMathOperator{\fl}{\mathsf{fl}}
\DeclareMathOperator{\tr}{tr}
\DeclareMathOperator{\brank}{\overline{\rank}}
\DeclareMathOperator{\rank}{rank}

\newcommand{\op}{\mathbin{\mathsf{op}}}

\newcommand{\tp}{{\scriptscriptstyle\mathsf{T}}}
\newcommand{\F}{{\scriptscriptstyle\mathsf{F}}}
\newcommand{\St}{{\scriptscriptstyle\mathsf{S}}}
\newcommand{\W}{{\scriptscriptstyle\mathsf{W}}}
\newcommand{\C}{{\scriptscriptstyle\mathsf{C}}}
\newcommand{\G}{{\scriptscriptstyle\mathsf{G}}}
\newcommand{\R}{{\scriptscriptstyle\mathsf{R}}}
\newcommand{\N}{{\scriptscriptstyle\mathsf{N}}}

\usepackage{algorithm}
\usepackage{algpseudocode}

\usepackage{neuralnetwork}

\theoremstyle{definition}
\newtheorem{theorem}{Theorem}[section]
\newtheorem{definition}[theorem]{Definition}
\newtheorem{corollary}[theorem]{Corollary}
\newtheorem{proposition}[theorem]{Proposition}

\numberwithin{equation}{section}

\begin{document}
\title{Numerical stability and tensor nuclear norm}
\author{Zhen Dai}
\author{Lek-Heng Lim}
\address{Computational and Applied Mathematics Initiative,
University of Chicago, Chicago, IL 60637-1514.}
\email{zhen9@uchicago.edu}
\email{lekheng@uchicago.edu}

\begin{abstract}
We present a notion of bilinear stability, which is to numerical stability what bilinear complexity is to time complexity. In bilinear complexity, an algorithm for evaluating a bilinear operator $\beta : \mathbb{U} \times \mathbb{V} \to \mathbb{W}$ is a decomposition $\beta =  \varphi_1 \otimes \psi_1 \otimes w_1 + \dots + \varphi_r \otimes \psi_r \otimes w_r $; the number of terms $r$ captures the speed of the algorithm; and its smallest possible value, i.e., the tensor rank of $\beta$, quantifies the speed of a fastest algorithm. Bilinear stability introduces norms to the mix: The growth factor of the algorithm $\lVert \varphi_1 \rVert_*  \lVert \psi_1 \rVert_* \lVert w_1 \rVert + \dots + \lVert \varphi_r \rVert_* \lVert \psi_r \rVert_* \lVert w_r \rVert$ captures the accuracy of the algorithm; and its smallest possible value, i.e., the tensor nuclear norm of $\beta$, quantifies the accuracy of a stablest algorithm. To substantiate this notion, we establish a bound for the forward error in terms of the growth factor and present numerical evidence comparing various fast algorithms for matrix and complex multiplications, showing that larger growth factors correlate with less accurate results. Compared to similar studies of numerical stability, bilinear stability is more general, applying to any bilinear operators and not just matrix or complex multiplications; is more simplistic, bounding forward error in terms of a single (growth) factor; and is truly tensorial like bilinear complexity, invariant under any orthogonal change of coordinates. 
As an aside, we study a new algorithm for computing complex multiplication in terms of real, much like Gauss's, but is optimally fast and stable in that it attains both tensor rank and nuclear norm.
\end{abstract}

\maketitle

\section{Introduction}

More than fifty years ago, in Volume~13 of this journal, Volker Strassen announced an astounding result: A pair of $2 \times 2$ matrices may be multiplied with seven multiplications \cite{strassen}. A consequence is that linear systems can be solved in $O(n^{\log_2 7})$ time complexity, a surprise at that time as existing works such as \cite{KK} purportedly showed that $O(n^3)$ was the lowest possible.

Strassen's algorithm is in the spirit of the well-known algorithm, often attributed to Gauss,\footnote{See \cite[p.~37]{Moore}, \cite[p.~8]{Rudich} for example.} for multiplying a pair of complex numbers with three real multiplications \cite{Higham},
\begin{equation}\label{eq:gauss}
(a + bi)(c+ di) =  (ac - bd) + i[(a+b)(c+d) -ac -bd],
\end{equation}
but is notable in that Strassen's applies to a noncommutative product (matrix multiplication) as opposed to a commutative one (complex scalar multiplication). It led to a plethora of followed-up works and ultimately to the realization that there is a unified framework underlying the algorithms of Gauss and Strassen, namely, in evaluating a bilinear operator $\beta \colon \mathbb{U} \times \mathbb{V} \to \mathbb{W}$, viewed as a $3$-tensor in $\mathbb{U}^* \otimes \mathbb{V}^* \otimes \mathbb{W}$, any decomposition
\begin{equation}\label{eq:decomp}
\beta = \varphi_1 \otimes \psi_1 \otimes w_1 + \dots + \varphi_r \otimes \psi_r \otimes w_r
\end{equation}
into linear functionals $\varphi_i \colon \mathbb{U} \to \mathbb{R}$, $\psi_i : \mathbb{V} \to \mathbb{R}$, and vectors $w_i \in \mathbb{W}$, $i=1,\dots,r$, gives us an algorithm for computing $\beta$. Furthermore, the number of terms $r$ in such a decomposition counts precisely the number of multiplications, and thus the minimal value of $r$, i.e., the \emph{tensor rank} of $\beta$, gives the optimal complexity for evaluating $\beta$  in an appropriate sense \cite{Crelle} (see Section~\ref{sec:speed}). Both Gauss's and Strassen's algorithms are the fastest possible according to this measure, that is, they attain the tensor ranks of complex multiplication (three) and $2 \times 2$ matrix product (seven) respectively  \cite{Winograd-sharp}. 

Well-known to readers of this journal, speed is not all that matters in an algorithm, numerical stability is arguably more important in finite-precision computations as rounding errors may result in an unstable algorithm producing no correct digits. While the stability of algorithms for evaluating bilinear operators has been studied for specific algorithms or operators in isolation, e.g., for Gauss's algorithm in \cite{Higham}, Strassen's algorithm in \cite{Brent}, and other fast matrix multiplication algorithms in \cite{Ballard,Bini}, there has been no unfied treatment that applies to all bilinear operators $\beta$ as in the case of speed. There is no analysis that quantifies stability in terms of some tensorial property of $\beta$ analogous to how speed is quantified in terms of its tensor rank. The goal of the present article is to fill this gap. We will show that just as the number of terms $r$ in the decomposition \eqref{eq:decomp} controls the speed of the algorithm, the \emph{growth factor}, defined as
\begin{equation}\label{eq:growth}
\lVert \varphi_1 \rVert_*  \lVert \psi_1 \rVert_* \lVert w_1 \rVert + \dots + \lVert \varphi_r \rVert_* \lVert \psi_r \rVert_* \lVert w_r \rVert,
\end{equation}
controls the stability of the algorithm; and just as the tensor rank of $\beta$ measures the optimal speed, the \emph{tensor nuclear norm} of $\beta$, defined as
\begin{equation}\label{eq:nuclear}
\lVert \beta \rVert_{\nu} \coloneqq \inf\biggl\{  \sum_{i=1}^r \lVert\varphi_i \rVert_* \lVert\psi_i \rVert_* \lVert w_i\rVert : \beta = \sum_{i=1}^{r} \varphi_i \otimes \psi_i\otimes w_i\biggr\},
\end{equation}
measures the optimal stability, the precise meaning of which we will state in due course.

Although we have alluded to the relation between tensor nuclear norm and numerical stability in earlier works \cite{acta,struct,nuclear}, we have never stated a precise relation nor carried out numerical experiments to demonstrate the relation. This article provides both. Theorem~\ref{thm:main} gives a general relation between the growth factor of a bilinear algorithm and its forward error, from which a relation between tensor nuclear norm and forward error may be deduced as in Corollary~\ref{cor:main}. We then perform a range of numerical experiments involving Gauss's and Strassen's algorithms to substantiate our theoretical findings:
\begin{description}
\item[Matrix multiplication] We compare Strassen's algorithm with a well-known variant due to Winograd \cite{HighamBook,winograd}. While both attain the optimal seven multiplications, Winograd's variant is often favored because it requires only fifteen additions, compared to Strassen's eighteen. Nevertheless we will show that Strassen's algorithm has a growth factor of $12 + 2\sqrt{2} \approx 14.83$ whereas Winograd's variant has a growth factor of $7 + 4\sqrt{2} +3 \sqrt{3} \approx 17.85$. For comparison, the conventional algorithm for $2 \times 2$ matrix product has eight multiplications and a growth factor of $8$. Our numerical experiments confirm that  in terms of accuracy Winograd's is indeed worse than Strassen's, which is in turn worse than the conventional algorithm, as Theorem~\ref{thm:main} indicates.

\item[Complex multiplication] We compare the regular algorithm for complex multiplication, which requires four real multiplications and has a growth factor of $4$; Gauss's algorithm,  which requires three real multiplications but has a larger growth factor of  $2(1+\sqrt{2}) \approx 4.83$; and a new algorithm:
\begin{equation}\label{eq:ournew}
\begin{aligned}
(a + bi)(c+ di) &= \frac{1}{2} \biggl[
\biggl(a + \frac{1}{\sqrt{3}}b\biggr)
\biggl(c + \frac{1}{\sqrt{3}}d\biggr)
+\biggl(a - \frac{1}{\sqrt{3}}b\biggr)
\biggl(c -  \frac{1}{\sqrt{3}}d\biggr)
-\frac{8}{3}bd
\biggr]\\
&\qquad+
\frac{i\sqrt{3}}{2} \biggl[
\biggl(a + \frac{1}{\sqrt{3}}b\biggr)
\biggl(c + \frac{1}{\sqrt{3}}d\biggr)
-\biggl(a - \frac{1}{\sqrt{3}}b\biggr)
\biggl(c - \frac{1}{\sqrt{3}}d\biggr)
\biggr].
\end{aligned}
\end{equation}
This new algorithm has the best features of both the regular and Gauss's algorithms, requiring three real multiplications and yet has the smaller (in fact, smallest, as we will see) growth factor of $4$. Again the results are consistent with the prediction of Theorem~\ref{thm:main}.
\end{description}

For the uninitiated, we would like to stress that the aforementioned algorithms only begin to make a difference when they are applied recursively, or applied to matrices, or both. For instance, Gauss's algorithm \eqref{eq:gauss} is really quite useless for multiplying a pair of complex numbers, whether `by hand' or on a computer. It only becomes useful when applied recursively in the form of Karatsuba's algorithm \cite{Kara} for integer multiplication, with $i$ replaced by the number base; or when applied to complex matrices \cite{Fam}:
\begin{equation}\label{eq:G}
(A + iB)(C+ iD)= (AC - BD) + i[(A+B)(C+D) -AC -BD],
\end{equation}
with $A + i B, C + i D \in \mathbb{C}^{n \times n}$, $A, B, C, D \in \mathbb{R}^{n \times n}$. As multiplication of matrices is much more expensive than addition of matrices, so \eqref{eq:G} really does represent an enormous savings in speed over the regular algorithm:
\begin{equation}\label{eq:U}
(A + iB)(C+ iD)= (AC - BD) + i(BC +AD).
\end{equation}
Likewise, our new algorithm \eqref{eq:ournew} only begins to make a difference when applied to matrices. For the same reason,  the algorithms of Strassen and Winograd are only worth the trouble when applied recursively to a product of $n \times n$ matrices partitioned recursively into $2 \times 2$ blocks.

To address another related point early on, a surprisingly common complaint among early feedbacks is that there are a lot of $\sqrt{3}$'s in our algorithm \eqref{eq:ournew}. Certainly, if one computes these products `by hand,' it would be easier to use the regular or Gauss's algorithm since they do not involve irrational coefficients. But when performed by a computer this is completely immaterial. In case it is not clear, it does not matter whether we multiply by $3$ or by $\sqrt{3}$; to a computer (or any IEEE 754-compliant equipment) both are binary strings of $0$'s and $1$'s and arithmetic takes one flop regardless. Maybe there would be some minor savings when a constant happens to be a power of $2$ --- because of binary arithmetic --- but aside from that, it makes no difference what coefficients appear in our algorithm.

For the matrix multiplication experiments, our goal is to illustrate Theorem~\ref{thm:main} by comparing the known algorithms of Strassen and Winograd. Incidentally, a numerical comparison of the accuracy of Strassen's algorithm and Winograd's variant was stated as a research problem in \cite[Exercise~23.10]{Higham}. Our work in Section~\ref{sec:fmm} supplies both numerical evidence and a rigorous explanation of why  Strassen’s is more accurate than Winograd’s.

For the complex multiplication experiments, aside from providing another illustration of Theorem~\ref{thm:main}, we also have the additional goal of testing, for the first time, the new algorithm \eqref{eq:ournew} applied to multiply complex matrices, which we will see is
\begin{itemize}
\item nearly as fast as Gauss's algorithm \eqref{eq:G}, and
\item nearly as stable as the regular algorithm \eqref{eq:U}.
\end{itemize}
To substantiate these claims, we perform more extensive experiments to compare \eqref{eq:ournew}, \eqref{eq:G}, and \eqref{eq:U}, including three practical applications: evaluation of matrix polynomials via Horner's method \cite{matrix_function}, unitary transform, and complex-valued neural networks \cite{complexNN_1,complexNN_6,complexNN_3,complexNN_4,complexNN_5,complexNN_2}. All our codes are available from \url{https://github.com/zhen06/Complex-Matrix-Multiplication}.

\subsection*{Conventions}

To reduce notational clutter, we denote norms on different vector spaces $\mathbb{U}, \mathbb{V}, \mathbb{W}$ by the same $\lVert \, \cdot \, \rVert$. There is no cause for confusion since we always use it in a form like $\lVert v \rVert$ for some $v \in \mathbb{V}$, where it is clear from context that   $\lVert \, \cdot \, \rVert$ refers to a norm on $\mathbb{V}$. Likewise the corresponding dual norms on $\mathbb{U}^*, \mathbb{V}^*, \mathbb{W}^*$ will be denoted by the same $\lVert \, \cdot \, \rVert_*$. Recall that for $\varphi \in \mathbb{V}^*$, i.e., $\varphi : \mathbb{V} \to \mathbb{R}$ is a linear functional, this is defined by
\[
\lVert \varphi \rVert_* \coloneqq \sup \{\lvert \varphi(v) \rvert : \lVert v \rVert \le 1\}.
\]
In this article, ``stability'' and ``accuracy'' have the same meaning, i.e.,  small forward error, but the former is used to describe an algorithm whereas the latter is used to describe its output.

\section{Bilinear complexity}\label{sec:speed}

We provide a brief review of bilinear complexity, usually studied in Algebraic Computational Complexity \cite{Borodin,BCS,Landsberg,Handbook}, for numerical analysts. Our goals here are to (i) highlight certain departures from typical practice in numerical linear algebra; and (ii) show a parallel with our notion of bilinear stability in the next section.

Let $\mathbb{U}$, $\mathbb{V}$, $\mathbb{W}$ be finite-dimensional vector spaces, assume to be over $\mathbb{R}$ for simplicity. Let $\beta : \mathbb{U} \times \mathbb{V} \to \mathbb{W}$ be a bilinear operator. Depending on one's definition of a tensor, we have $\beta \in \mathbb{U}^* \otimes \mathbb{V}^* \otimes \mathbb{W}$ either through definition \cite[Definition~3.3]{acta} or by the universal mapping property \cite[Equation~4.88]{acta}. A \emph{bilinear algorithm} for evaluating $\beta$ is a decomposition of the form \eqref{eq:decomp}. In other words, for any $u \in \mathbb{U}$ and $v \in \mathbb{V}$, we evaluate $\beta(u,v)$ by performing the algorithm given by the decomposition on the right:
\begin{equation}\label{eq:decompbilin2a}
\beta(u,v) = \sum_{i=1}^{r} \varphi_i (u) \psi_i (v) w_i.
\end{equation}
In practice, the vector spaces involved are usually Euclidean spaces of vectors $\mathbb{R}^n$ or  matrices $\mathbb{R}^{m \times n}$. Riesz representation theorem guarantees that any linear functional $\varphi : \mathbb{R}^n \to \mathbb{R}$ must take the form $\varphi(x) = a^\tp x$ for some $a \in \mathbb{R}^n$ and likewise any functional $\varphi : \mathbb{R}^{m \times n} \to \mathbb{R}$ must take the form $\varphi(X) = \tr(A^\tp X)$ for some $A \in \mathbb{R}^{m \times n}$.

Each rank-one term $\varphi_i (u) \psi_i (v) w_i$ in \eqref{eq:decompbilin2a} accounts for one multiplication but herein lies a pitfall --- the `multiplication' refers to the product  of $\varphi_i(u)$ and $\psi_i(v)$; note that this a variable product, i.e., the value depends on variables $u$ and $v$, as opposed to a scalar product. Take a randomly made-up example\footnote{Genuine examples to follow in Sections~\ref{sec:fmm} and \ref{sec:cplx}.} with $\mathbb{U} = \mathbb{R}^{2 \times 2}$, $\mathbb{V} =\mathbb{R}^2$,  $\mathbb{W} = \mathbb{R}^3$, and
\[
\varphi_i\bigl(\begin{bsmallmatrix} a & b \\ c & d \end{bsmallmatrix} \bigr) = \tr\bigl(\begin{bsmallmatrix} -1 & 0 \\ 1 & 2 \end{bsmallmatrix}^\tp \begin{bsmallmatrix} a & b \\ c & d \end{bsmallmatrix} \bigr)  = -a + c + 2d, \quad \psi_i\bigl(\begin{bsmallmatrix} x \\ y \end{bsmallmatrix}\bigr) =\begin{bsmallmatrix} 3 \\ -1/2 \end{bsmallmatrix}^\tp \begin{bsmallmatrix} x \\ y \end{bsmallmatrix}  = 3x - y/2, \quad w_i  =\begin{bsmallmatrix} -3 \\ 4 \\ \sqrt{5} \end{bsmallmatrix},
\]
then there is exactly one multiplication in
\[
\varphi_i (u) \psi_i (v) w_i =  \begin{bsmallmatrix} -3( -a + c + 2d)(3x - y/2) \\  4( -a + c + 2d)(3x -y/2)   \\ \sqrt{5}( -a + c + 2d)(3x -y/2) \end{bsmallmatrix}.
\]
The scalar products like $2d$ or $-y/2$ or $\sqrt{5}t $ are discounted in Strassen's model of bilinear complexity \cite{Crelle,Crelle2} and for good reasons --- these constants coefficients are fixed in the algorithm and can be hardcoded or hardwired, unlike the product between $ -a + c + 2d$ and $3x -y/2$, which depends on the variable inputs $u =\begin{bsmallmatrix} a & b \\ c & d \end{bsmallmatrix} $ and $v = \begin{bsmallmatrix} x \\ y \end{bsmallmatrix}$. In particular, Strassen's measure of speed, called \emph{bilinear complexity}, is independent of the values of these constant coefficients, but we will show in the next section that these will affect numerical stability of the algorithm.

To emphasize its distinction from scalar products, Strassen calls a variable product in the above sense a \emph{nonscalar product} \cite{Crelle2}. In other words, bilinear complexity measures speed purely in terms of the number of nonscalar products. The bilinear complexity of the algorithm in \eqref{eq:decompbilin2a}  is given by the number terms in the decomposition $r$ and the optimal speed of evaluating $\beta$ is therefore given by the tensor rank \cite{Crelle}
\begin{equation}\label{eq:trankbilin}
\rank(\beta) \coloneqq \min\biggl\{r : \beta = \sum_{i=1}^{r} \varphi_i \otimes \psi_i \otimes w_i\biggr\}.
\end{equation}
A \emph{tensor rank decomposition} of $\beta$, i.e., one that attains its tensor rank, is then a fastest algorithm in the context of bilinear complexity.

In realistic scenarios, storage and computations both have finite-precision. Given $u$ and $v$, we do not need to know $\beta(u,v)$ exactly; in fact computing anything beyond $16$ decimal digits of accuracy is wasted effort since we do not store more than $16$ digits in IEEE double precision. So the tensor rank of $\beta$ is less relevant than the \emph{border rank} \cite{border} of $\beta$, which is the smallest $r$ so that
\[
\lVert \beta -  \varphi_1^\varepsilon \otimes \psi_1^\varepsilon \otimes w_1^\varepsilon - \varphi_2^\varepsilon \otimes \psi_2^\varepsilon \otimes w_2^\varepsilon - \dots - \varphi_r^\varepsilon \otimes \psi_r^\varepsilon \otimes w_r^\varepsilon \rVert < \varepsilon
\]
for all $\varepsilon > 0$, or, formally,
\begin{equation}\label{eq:brankbilin}
\brank(\beta) \coloneqq \min\biggl\{r : \beta = \lim_{\varepsilon \to 0^+ } \sum_{i=1}^{r} \varphi_i^\varepsilon \otimes \psi_i^\varepsilon \otimes w_i^\varepsilon\biggr\}.
\end{equation}

For the two problems studied in our article, namely, matrix multiplication,
\[
\beta_{m,n,p} : \mathbb{R}^{m \times n} \times  \mathbb{R}^{n \times p} \to  \mathbb{R}^{m \times p}, \quad (A,B) \mapsto AB,
\]
and complex multiplication,
\[
\beta_\mathbb{C} :\mathbb{C} \times \mathbb{C} \to \mathbb{C}, \quad (w,z) \mapsto wz,
\]
(noting that $\mathbb{C}$ is a two-dimensional real vector space), we have \cite{Landsberg2x2,Winograd-sharp}
\[
\rank(\beta_{2,2,2}) = \brank(\beta_{2,2,2}) = 7, \qquad \rank(\beta_\mathbb{C}) = \brank(\beta_\mathbb{C}) = 3.
\]
It is in general difficult to find such exact values. For instance, the values of $\rank(\beta_{3,3,3})$ and $\brank(\beta_{3,3,3})$ are still unknown. Most of the efforts in studying matrix multiplication go towards determining the asymptotic value $\omega \coloneqq \inf\{ p \in\mathbb{R} \colon  \rank(\beta_{n,n,n})=O(n^p)\}$, called the \emph{exponent of matrix multiplication}. An advantage is that asymptotically, the full arithmetic complexity, i.e., counting all operations and not just nonscalar multiplications, is also $O(n^\omega)$. More importantly, the role of $\omega$ stretches far beyond matrix multiplication, governing the full arithmetic complexity of computing inverse, determinant, null basis, linear systems, LU/QR/eigenvalue/Hessenberg decompositions, characteristic polynomials, sparsification, and even linear programming --- note in particular that none of these are bilinear operations \cite{Crelle2} (see also \cite[Chapter~16]{BCS} and \cite[Examples~3.10 and 4.40]{acta}.

\section{Bilinear stability}\label{sec:growth}

We would like to state at the outset that numerical stability is a moderately complicated issue that depends on many factors and cannot be completely represented by any single number. Designing numerically stable algorithms is as much an art as it is a science. However the six \emph{Higham guidelines} for numerical stability \cite[Section~1.18]{HighamBook} capture the most salient aspects. Among them, the second guideline to ``minimize the size of intermediate quantities relative to the final solution'' is one of the most unequivocal, lends itself to precise quantification, and is what we will focus on in this section. Consideration of Higham's second guideline for bilinear algorithms leads us naturally to the notion of \emph{bilinear stability}, which relates to accuracy the way bilinear complexity relates to speed. More precisely, the growth factor \eqref{eq:growth} and tensor nuclear norm \eqref{eq:nuclear} are to accuracy in bilinear stability what the number of rank-$1$ terms in \eqref{eq:decompbilin2a} and the tensor rank \eqref{eq:trankbilin} are to speed in bilinear complexity. Here accuracy refers to the size of relative forward error.

Bilinear stability differs from existing studies of numerical stability of bilinear algorithms such as those in \cite{Ballard,Bini,Brent,Higham} in three ways: (i) it is more general, applying to any bilinear operators as opposed to specific ones like matrix multiplication; (ii) it is more simplistic, relating forward error to just growth factor as opposed to two or three different factors in the approaches of \cite{Ballard,Bini}; (iii) it is truly tensorial, as growth factor and tensor nuclear norm are invariant under any orthogonal change-of-coordinates, just as tensor rank is invariant under any invertible change-of-coordinates. The factors (i) and (ii), i.e., generality and simplicity, may often be sacrificed for better bounds: Given any specific bilinear operator, we may often obtain smaller forward error bounds by performing a more precise analysis tailored to that given operator. We will do see this in Section~\ref{sec:roundcplx}.

One difference between bilinear complexity and bilinear stability is that the latter requires a norm. While there are many excellent treatises on tensor norms \cite{Defant, Diestel, Ryan}, they are excessive for our purpose. All the reader needs to know is that for a vector space $\mathbb{V}_i$ with norm $\lVert \, \cdot \, \rVert_i$, $i =1,\dots,d$, a tensor norm $\lVert \, \cdot \, \rVert$ on $\mathbb{V}_1 \otimes \mathbb{V}_2 \otimes \dots \otimes \mathbb{V}_d$ satisfies the multiplicativity property for rank-$1$ tensors:
\[
\lVert v_1 \otimes v_2 \otimes \dots \otimes v_d \rVert = \lVert v_1 \rVert_1 \lVert v_2 \rVert_2 \cdots \lVert v_d\rVert_d,
\]
where $v_i \in \mathbb{V}_i$. In particular, the spectral, Frobenius (also called Hilbert--Schmidt), nuclear norms \cite[p.~561 and Example~4.17]{acta} are all equal on rank-$1$ tensors in $\mathbb{U}^* \otimes \mathbb{V}^* \otimes \mathbb{W}$, i.e., 
\[
\lVert \varphi \otimes \psi \otimes w \rVert_\sigma = \lVert \varphi \otimes \psi \otimes w \rVert_\F = \lVert \varphi \otimes \psi \otimes w \rVert_\nu = \lVert \varphi \rVert_* \lVert \psi \rVert_* \lVert w\rVert
\]
for all $\varphi \in \mathbb{U}^*$, $\psi_i \in \mathbb{V}^*$, $w \in \mathbb{W}$. Consequently, when we speak of the norm of a rank-$1$ tensor  $ \varphi \otimes \psi \otimes w$, it does not matter which of these three norms we choose, and we will simply write
\[
\lVert \varphi \otimes \psi \otimes w \rVert \coloneqq \lVert \varphi \rVert_* \lVert \psi \rVert_* \lVert w\rVert.
\]

We first present a straightforward heurstic that motivates our definition of the growth factor, deferring the more formal forward error analysis to Theorem~\ref{thm:main}. If we apply the rank-one bilinear operator $\varphi_i \otimes \psi_i \otimes w_i$ to $u$ and $v$,
\begin{align*}
\lVert (\varphi_i \otimes \psi_i \otimes w_i) (u,v)\rVert = \lVert \varphi_i (u) \psi_i (v) w_i \rVert &= \lvert \varphi_i (u) \rvert \lvert \psi_i (v) \rvert \lVert w_i \rVert \\
&\le  \lVert \varphi_i \rVert_* \lVert u \rVert \lVert \psi_i \rVert_*  \lVert v\rVert \lVert w_i\rVert = \lVert \varphi_i \otimes \psi_i \otimes w_i \rVert  \lVert u\rVert \lVert v\rVert.
\end{align*}
So $\varphi_i \otimes \psi_i \otimes w_i$ magnifies the errors in $u$ and $v$ by an amount bounded by its tensor norm $\lVert \varphi_i \otimes \psi_i \otimes w_i\rVert$. 
Therefore, in a bilinear algorithm given by the right side of \eqref{eq:decompbilin2a} for evaluating $\beta$, triangle inequality gives
\[
\lVert \beta(u,v) \rVert = \biggl\lVert \sum_{i=1}^{r}  (\varphi_i \otimes \psi_i \otimes w_i) (u,v) \biggr\rVert \le   \biggl[ \sum_{i=1}^{r}   \lVert \varphi_i \otimes \psi_i \otimes w_i \rVert \biggr]  \lVert u\rVert \lVert v\rVert.
\]
The  $i$th step of the algorithm magnifies the error in the inputs $(u,v)$ by an amount bounded by  $\lVert \varphi_i \otimes \psi_i \otimes w_i\rVert$ and over the course of $r$ steps in the algorithm, the accumulated error is bounded by a factor of
\begin{equation} \label{eq:growthfactor}
 \sum_{i=1}^{r}   \lVert \varphi_i \otimes \psi_i \otimes w_i \rVert = \sum_{i=1}^{r} \lVert \varphi_i \rVert_*\lVert \psi_i \rVert_*\lVert w_i \rVert,
\end{equation}
which we will define as the \emph{growth factor} of the algorithm or decomposition \eqref{eq:decompbilin2a}. Its minimum value over all possible bilinear algorithms for evaluating $\beta$ or, equivalently, over all decomposition of $\beta$ as a $3$-tensor is therefore given by the nuclear norm \eqref{eq:nuclear}. This idea was first floated in \cite[Section~3.2]{struct}. Note that the growth factor depends on the algorithm/decomposition for $\beta$ but the nulcear norm depends only on $\beta$.

We now state a formal definition to make precise the terms used in the preceding discussions.
\begin{definition}
Let $\mathbb{U}, \mathbb{V}, \mathbb{W}$ be three finite-dimensional real vector spaces. A  \emph{decomposition} of a bilinear operator $\beta: \mathbb{U} \times \mathbb{V} \to \mathbb{W}$  is a list $\mathcal{D} = (\varphi_i,\psi_i,w_i)_{i = 1}^r $ with
\begin{equation}\label{eq:betadecomp}
\beta = \sum_{i=1}^{r} \varphi_i \otimes \psi_i \otimes w_i,
\end{equation}
where $\varphi_i : \mathbb{U} \to \mathbb{R}$ and $\psi_i : \mathbb{V} \to \mathbb{R}$ are linear functionals  and $w_i \in \mathbb{W}$, $i = 1,\dots,r$. An \emph{algorithm} $\widehat{\beta}_{\mathcal{D}}$ given by the decomposition $\mathcal{D}$ takes $(u,v) \in \mathbb{U} \times \mathbb{V}$ as inputs and computes the output $\beta(u,v)$ in three steps:
\begin{enumerate}[\upshape (i)]
    \item computes $\varphi_i(u)$ and $\psi_i(v)$, $i = 1,\dots,r$;
    \item\label{it:varmult} computes $\varphi_i(u) \psi_i(v) w_i$, $i = 1,\dots,r$;
    \item computes $\sum_{i = 1}^r \varphi_i(u) \psi_i(v) w_i$.
\end{enumerate}
The \emph{growth factor} of the algorithm $\widehat{\beta}_{\mathcal{D}}$ is defined as
\[
    \gamma(\widehat{\beta}_{\mathcal{D}}) \coloneqq \sum_{i=1}^{r} \lVert \varphi_i \otimes \psi_i \otimes w_i \rVert = \sum_{i=1}^{r} \lVert \varphi_i \rVert_*\lVert \psi_i \rVert_*\lVert w_i \rVert.
\]
\end{definition}
As noted in Section~\ref{sec:speed}, only the variable multiplication in step~\eqref{it:varmult} counts  in bilinear complexity; the other two steps comprising scalar multiplications and additions are discounted. In bilinear stability all three steps contribute to the growth factor.
\begin{proposition}\label{prop:nuclear}
The minimal growth factor is given by nucler norm of the $\beta$, i.e.,
\[
\min_\mathcal{D}     \gamma(\widehat{\beta}_{\mathcal{D}}) = \lVert  \beta \rVert_{\nu},
\]
with $\mathcal{D}$ running over all decomposition. Furthermore,  there is always an algorithm that attains the minimal growth factor.
\end{proposition}
The above equality is just stating \eqref{eq:nuclear}  in terms of the growth factor. That there is always an algorithm attaining the minimal growth factor, justifying our writing $\min$ instead of $\inf$, follows from the existence of a \emph{nuclear decomposition} \cite[Proposition~3.1]{nuclear}, i.e., a decomposition that attains the nuclear norm. Just as a rank decomposition of $\beta$ represents a fastest algorithm in bilinear complexity, a nuclear decomposition of $\beta$ represents a stablest algorithm in bilinear stability.

We next establish a rigorous relationship between growth factor and numerical stability by proving a forward error bound in terms of the growth factor of a bilinear algorithm. We assume a system of floating point arithmetic obeying the standard model as in \cite{HighamBook}: For $x,y \in \mathbb{R}$
\begin{equation}\label{eq:float}
    \fl(x \op y) = (x \op y)(1+\delta), \quad \quad |\delta| \le \boldsymbol{\mathsf{u}}, \quad \op = +,-,*,/
\end{equation}
with $\boldsymbol{\mathsf{u}}$ the unit roundoff, except when $\fl(x \op y) = 0$, in which case $\delta$ becomes $-1$.
We assume that $\mathbb{U}, \mathbb{V}, \mathbb{W}$ are vector spaces of dimensions $m,n,p$ and that appropriate computational bases have been chosen on them so that we may identify $\mathbb{U} \cong \mathbb{R}^m$, $\mathbb{V} \cong \mathbb{R}^n$, $\mathbb{W} \cong \mathbb{R}^p$. The computational bases do not need to be the standard bases and may instead be Fourier, Krylov, Haar, wavelet bases, etc. This is another reason why we cast our discussions in terms of abstract vector spaces and do not choose bases until absolutely necessary. However, once a choice of bases has been made, the result below depends only on the dimensions of $\mathbb{U}, \mathbb{V}, \mathbb{W}$; if say, $\mathbb{U} = \mathbb{R}^{m \times n}$, then only the fact that it has dimension $mn$ matters, i.e., $\mathbb{U} \cong \mathbb{R}^{mn}$.

\begin{theorem}[Growth factor and forward error]\label{thm:main}
Let $\beta: \mathbb{R}^m \times \mathbb{R}^n \to \mathbb{R}^p$ be a bilinear operator, $\mathcal{D} = (\varphi_i,\psi_i,w_i)_{i = 1}^r$ a decomposition, and $\widehat{\beta}_{\mathcal{D}}$ the corresponding algorithm. If $\widehat{\beta}_{\mathcal{D}}(u,v)$ is the output of $\widehat{\beta}_{\mathcal{D}}$ computed using floating point operations, with  $u \in \mathbb{R}^m$ and $v \in \mathbb{R}^n$ as inputs, then
\begin{equation*}
    \lVert \beta(u,v) - \widehat{\beta}_{\mathcal{D}}(u,v) \rVert_\infty \le (m+n+r+1) \gamma(\widehat{\beta}_{\mathcal{D}}) \lVert u \rVert \lVert v \rVert \boldsymbol{\mathsf{u}} + O(\boldsymbol{\mathsf{u}}^2).
\end{equation*}
\end{theorem}
\begin{proof}
We first show that the result reduces to the case $p = 1$. It suffices to show that
\begin{equation} \label{pf:gf:eq:red}
\lvert \beta(u,v)_k - \widehat{\beta}_{\mathcal{D}}(u,v)_k\rvert \le (m+n+r+1) \gamma(\widehat{\beta}_{\mathcal{D}}) \lVert u \rVert \lVert v \rVert \boldsymbol{\mathsf{u}} + O(\boldsymbol{\mathsf{u}}^2)
\end{equation}
for all $k = 1, \dots, p$, where the subscript $k$ refers to the $k$th coordinate of a vector in $\mathbb{R}^p$. Since
\[
\gamma(\widehat{\beta}_{\mathcal{D}}) = \sum_{i=1}^{r} \lVert \varphi_i \rVert_*\lVert \psi_i \rVert_*\lVert w_i \rVert \ge \sum_{i=1}^{r} \lVert \varphi_i \rVert_*\lVert \psi_i \rVert_* |w_{ik}|,
\]
with $w_{ik}$ the $k$th coordinate of $w_k \in \mathbb{R}^p$, to show \eqref{pf:gf:eq:red}, it suffices to show
\[
|\beta(u,v)_k - \widehat{\beta}_{\mathcal{D}}(u,v)_k| \le (m+n+r+1) \Biggl[\sum_{i=1}^{r} \lVert \varphi_i \rVert_*\lVert \psi_i \rVert_* |w_{ik}|\Biggr] \lVert u \rVert \lVert v \rVert \boldsymbol{\mathsf{u}} + O(\boldsymbol{\mathsf{u}}^2),
\]
which is equivalent to the case $p = 1$. In the following, we will assume that $p = 1$.

Since $\varphi_i$ and $\psi_i$ are linear functionals on $\mathbb{R}^m$ and $\mathbb{R}^n$, there exist $u_i \in \mathbb{R}^m$ and $v_i \in \mathbb{R}^n$ such that
\[
\varphi_i(u) = u_i^\tp  u \quad \textrm{and} \quad \psi_i(v) = v_i^\tp  v,
\]
for all $u \in \mathbb{R}^m$ and $v \in \mathbb{R}^n$. By \cite[equation~3.7]{HighamBook},
\[
|x^\tp y - \fl(x^\tp y)| \le n |x|^\tp |y| \boldsymbol{\mathsf{u}} + O(\boldsymbol{\mathsf{u}}^2),
\]
for any $x,y \in \mathbb{R}^n$ where  $|\cdot|$ applies coordinatewise. So for each $i = 1,\dots,r$,
\begin{equation} \label{pf:thm1:eq1}
\begin{split}
    |\varphi_i(u) - \fl(\varphi_i(u))| & = |u_i^\tp u - \fl(u_i^\tp u)| \le m |u_i|^\tp |u| \boldsymbol{\mathsf{u}} + O(\boldsymbol{\mathsf{u}}^2) \\
    &  \le m \lVert u_i \rVert \lVert u \rVert \boldsymbol{\mathsf{u}} + O(\boldsymbol{\mathsf{u}}^2) = m \lVert \varphi_i \rVert_* \lVert u \rVert \boldsymbol{\mathsf{u}} + O(\boldsymbol{\mathsf{u}}^2).
\end{split}
\end{equation}
Likewise, for each $i = 1,\dots,r$,
\begin{equation} \label{pf:thm1:eq2}
    |\psi_i(v) - \fl(\psi_i(v))| \le n \lVert \psi_i \rVert_* \lVert v \rVert \boldsymbol{\mathsf{u}} + O(\boldsymbol{\mathsf{u}}^2).
\end{equation}
Let $\Delta_{1,i} = \fl(\varphi_i(u)) - \varphi_i(u)$ and $\Delta_{2,i} = \fl(\psi_i(v)) - \psi_i(v)$. By \eqref{pf:thm1:eq1} and \eqref{pf:thm1:eq2},
\begin{equation} \label{pf:thm1:eq12}
    |\Delta_{1,i}| \le m \lVert \varphi_i \rVert_* \lVert u \rVert \boldsymbol{\mathsf{u}} + O(\boldsymbol{\mathsf{u}}^2), \quad
    |\Delta_{2,i}| \le n \lVert \psi_i \rVert_* \lVert v \rVert \boldsymbol{\mathsf{u}} + O(\boldsymbol{\mathsf{u}}^2).
\end{equation}
Let $c_i = \varphi_i(u)\psi_i(v)$ and $\widehat{c}_i$ be its computed value. By \eqref{pf:thm1:eq12}, there exists $\delta_i$ with $|\delta_i| \le \boldsymbol{\mathsf{u}}$ such that
\begin{equation} \label{pf:thm1:eq123}
\begin{split}
    \widehat{c}_i & = (\varphi_i(u) + \Delta_{1,i})(\psi_i(v)+\Delta_{2,i})(1+\delta_i) \\
    & = \varphi_i(u)\psi_i(v) + \Delta_{1,i} \psi_i(v) + \varphi_i(u) \Delta_{2,i} + \delta_i \varphi_i(u)\psi_i(v) + O(\boldsymbol{\mathsf{u}}^2).
\end{split}
\end{equation}
By \eqref{pf:thm1:eq12} and \eqref{pf:thm1:eq123},
\begin{equation} \label{pf:thm1:eq1234}
\begin{split}
    \quad |c_i - \widehat{c}_i|
    & \le m \lVert \varphi_i \rVert_* \lVert u \rVert |\psi_i(v)| \boldsymbol{\mathsf{u}} + |\varphi_i(u)| n \lVert \psi_i \rVert_* \lVert v \rVert \boldsymbol{\mathsf{u}} + |\varphi_i(u)\psi_i(v)| \boldsymbol{\mathsf{u}} + O(\boldsymbol{\mathsf{u}}^2)\\
    & \le (m+n+1) \lVert \varphi_i \rVert_* \lVert \psi_i \rVert_* \lVert u \rVert \lVert v \rVert \boldsymbol{\mathsf{u}} + O(\boldsymbol{\mathsf{u}}^2).
\end{split}
\end{equation}
Let $\Delta_i = \widehat{c}_i - c_i$.
By \eqref{pf:thm1:eq1234},
\begin{equation}  \label{pf:thm1:eq12345}
    |\Delta_i| \le (m+n+1) \lVert \varphi_i \rVert_* \lVert \psi_i \rVert_* \lVert u \rVert \lVert v \rVert \boldsymbol{\mathsf{u}} + O(\boldsymbol{\mathsf{u}}^2).
\end{equation}
Let $d_i = c_i w_i$ and $\widehat{d}_i$ be the computed value of $d_i$.
By \eqref{pf:thm1:eq12345}, there exists $\delta_i'$ with $|\delta_i'|\le \boldsymbol{\mathsf{u}}$ such that
\begin{equation}\label{pf:thm1:eq4}
    \widehat{d}_i = (c_i+\Delta_i)w_i(1+\delta_i') = c_i w_i + \Delta_i w_i + \delta_i' c_i w_i + O(\boldsymbol{\mathsf{u}}^2).
\end{equation}
Let $\Delta_i' = \widehat{d}_i - d_i$. By \eqref{pf:thm1:eq12345} and \eqref{pf:thm1:eq4},
\begin{equation} \label{pf:thm1:eq3}
\begin{split}
    \quad |\Delta_i'|
    & \le (m+n+1) \lVert \varphi_i \rVert_* \lVert \psi_i \rVert_* \lVert u \rVert \lVert v \rVert |w_i|  \boldsymbol{\mathsf{u}} + |\varphi_i(u)\psi_i(v)| |w_i| \boldsymbol{\mathsf{u}} + O(\boldsymbol{\mathsf{u}}^2) \\
    & \le (m+n+2) \lVert \varphi_i \rVert_* \lVert \psi_i \rVert_* |w_i| \lVert u \rVert \lVert v \rVert \boldsymbol{\mathsf{u}} + O(\boldsymbol{\mathsf{u}}^2).
\end{split}
\end{equation}
Finally, let $a = \sum_{i=1}^r\varphi_i(u)\psi_i(v)w_i$ and $\widehat{a}$ be the computed value of $a$.
By \eqref{pf:thm1:eq3}, there exists $\delta$ with $|\delta| \le \boldsymbol{\mathsf{u}}$ such that
\[
\begin{split}
    \widehat{a} & = \widehat{d}_1 (1+\delta)^{r-1} + \widehat{d}_2 (1+\delta)^{r-1} + \widehat{d}_3 (1+\delta)^{r-2} + \dots + \widehat{d}_r (1+\delta),
\end{split}
\]
where we compute the sum $\widehat{d}_1 + \widehat{d}_2 + \dots +\widehat{d}_r$ from left to right. Hence we obtain
\[
\begin{split}
    |a - \widehat{a}|
    & \le (m+n+2) \lVert u \rVert \lVert v \rVert \sum_{i = 1}^r \lVert \varphi_i \rVert_* \lVert \psi_i \rVert_* |w_i| \boldsymbol{\mathsf{u}} + (r-1)  \biggl \lVert \sum_{i = 1}^r c_i w_i \biggr \rVert  \boldsymbol{\mathsf{u}}+ O(\boldsymbol{\mathsf{u}}^2) \\
    & \le (m+n+r+1) \lVert u \rVert \lVert v \rVert \sum_{i = 1}^r \lVert \varphi_i \rVert_* \lVert \psi_i \rVert_* |w_i| \boldsymbol{\mathsf{u}} + O(\boldsymbol{\mathsf{u}}^2) \\
    & = (m+n+r+1) \gamma(\widehat{\beta}_{\mathcal{D}}) \lVert u \rVert \lVert v \rVert  \boldsymbol{\mathsf{u}}+ O(\boldsymbol{\mathsf{u}}^2). \qedhere
\end{split}
\]
\end{proof}

Theorem~\ref{thm:main} essentially says that that algorithms with small growth factors have small forward errors. Combined with Proposition~\ref{prop:nuclear}, we see that the optimally stable algorithm  in this context is the one corresponding to a nuclear decomposition of $\beta$.
\begin{corollary}[Tensor nuclear norm and forward error]\label{cor:main}
Let $\beta: \mathbb{R}^n \times \mathbb{R}^m \to \mathbb{R}^p$ be a bilinear operator, $\mathcal{D} = (\varphi_i,\psi_i,w_i)_{i = 1}^r$ a nuclear decomposition, and $\widehat{\beta}_{\mathcal{D}}$ the  corresponding algorithm. Then
\[
    \lVert \beta(u,v) - \widehat{\beta}_{\mathcal{D}}(u,v) \rVert_\infty \le (m+n+r+1) \lVert \beta \rVert_\nu \lVert u \rVert \lVert v \rVert \boldsymbol{\mathsf{u}}  + O(\boldsymbol{\mathsf{u}}^2).
\]
\end{corollary}

In principle, there is no reason to expect there to be an algorithm that is both fastest in the sense of Section~\ref{sec:speed} and stablest in the sense of this section, i.e., having a decomposition that attains both tensor rank and nuclear norm.  In Section~\ref{sec:cplx}, we will see that such an algorithm exists for complex multiplication and we will study its properties when applied to complex matrix multiplication.

\section{Fast matrix multiplications}\label{sec:fmm}

As an illustration of bilinear stability in the last section, we will calculate the growth factors of Strassen's algorithm \cite{strassen} and Winograd's variant \cite{HighamBook,winograd} for fast matrix multiplication and compare their stability empirically. We will see that the growth factor of Strassen's algorithm is smaller than that of Winograd's variant, and, consistent with the prediction of Theorem~\ref{thm:main}, numerical experiments indeed show that the former gives more accurate results. 

\subsection{Bilinear stability of Strassen multiplication}

Given two block matrices
\[
    A = \begin{bmatrix}
    A_{11} & A_{12} \\
    A_{21} & A_{22}
    \end{bmatrix},
    \quad
    B = \begin{bmatrix}
    B_{11} & B_{12} \\
    B_{21} & B_{22}
    \end{bmatrix},
\]
Strassen's algorithm \cite{strassen} first computes
\begin{alignat*}{2}
     M_1 &= (A_{11}+A_{22})(B_{11}+B_{22}), \qquad &
     M_5 &= (A_{11}+A_{12})B_{22}, \\
     M_2 &= (A_{21}+A_{22})B_{11}, &
     M_6 &= (A_{21}-A_{11})(B_{11}+B_{12}), \\
     M_3 &= A_{11}(B_{12}-B_{22}), &
     M_7 &= (A_{12}-A_{22})(B_{21}+B_{22}), \\
     M_4 &= A_{22}(B_{21}-B_{11}), &
\end{alignat*}
and then computes the product via
\[
AB =
    \begin{bmatrix}
    M_1+M_4-M_5+M_7 & M_3+M_5 \\
    M_2+M_4 & M_1-M_2+M_3+M_6
    \end{bmatrix}.
\]
Note that this may be applied recursively. Let $\widehat{\beta}_\St : \mathbb{R}^{2 \times 2} \times  \mathbb{R}^{2 \times 2} \to  \mathbb{R}^{2 \times 2} $ denote the Strassen's algorithm for $2 \times 2$ matrices. It is routine to check that for $A, B \in \mathbb{R}^{2 \times 2}$,
\[
    \widehat{\beta}_\St(A,B) = \sum_{i = 1}^7 \varphi_i(A)\psi_i(B)W_i,
\]
where $\varphi_i(A) = \tr(U_i^\tp  A)$ and $\psi_i(B) = \tr(V_i^\tp  B)$ with
\begin{alignat*}{4}
    U_1 &= 
    \begin{bmatrix}
    1 & 0 \\
    0 & 1
    \end{bmatrix},
    \quad
    & V_1 &= 
    \begin{bmatrix}
    1 & 0 \\
    0 & 1
    \end{bmatrix},
    \quad
    & W_1 &= 
    \begin{bmatrix}
    1 & 0 \\
    0 & 1
    \end{bmatrix}; \\
    U_2 &= 
    \begin{bmatrix}
    0 & 0 \\
    1 & 1
    \end{bmatrix},
    \quad
    & V_2 &= 
    \begin{bmatrix}
    1 & 0 \\
    0 & 0
    \end{bmatrix},
    \quad
    & W_2 &= 
    \begin{bmatrix}
    0 & 0 \\
    1 & -1
    \end{bmatrix};\\
    U_3 &= 
    \begin{bmatrix}
    1 & 0 \\
    0 & 0
    \end{bmatrix},
    \quad
    & V_3 &= 
    \begin{bmatrix}
    0 & 1 \\
    0 & -1
    \end{bmatrix},
    \quad
    & W_3 &= 
    \begin{bmatrix}
    0 & 1 \\
    0 & 1
    \end{bmatrix};\\
    U_4 &= 
    \begin{bmatrix}
    0 & 0 \\
    0 & 1
    \end{bmatrix},
    \quad
    & V_4 &= 
    \begin{bmatrix}
    -1 & 0 \\
    1 & 0
    \end{bmatrix},
    \quad
    & W_4 &= 
    \begin{bmatrix}
    1 & 0 \\
    1 & 0
    \end{bmatrix};\\
    U_5 &= 
    \begin{bmatrix}
    1 & 1 \\
    0 & 0
    \end{bmatrix},
    \quad
    & V_5 &= 
    \begin{bmatrix}
    0 & 0 \\
    0 & 1
    \end{bmatrix},
    \quad
    & W_5 &= 
    \begin{bmatrix}
    -1 & 1 \\
    0 & 0
    \end{bmatrix};\\
    U_6 &= 
    \begin{bmatrix}
    -1 & 0 \\
    1 & 0
    \end{bmatrix},
    \quad
    & V_6 &= 
    \begin{bmatrix}
    1 & 1 \\
    0 & 0
    \end{bmatrix},
    \quad
    & W_6 &= 
    \begin{bmatrix}
    0 & 0 \\
    0 & 1
    \end{bmatrix};\\
    U_7 &= 
    \begin{bmatrix}
    0 & 1 \\
    0 & -1
    \end{bmatrix},
    \quad
    & V_7 &= 
    \begin{bmatrix}
    0 & 0 \\
    1 & 1
    \end{bmatrix},
    \quad
    & W_7 &= 
    \begin{bmatrix}
    1 & 0 \\
    0 & 0
    \end{bmatrix}.
\end{alignat*}
For simplicity we will use the Frobenius norm on $\mathbb{R}^{2 \times 2}$ since it is self dual. The growth factor of Strassen's algorithm is then given by
\begin{equation}\label{strassen:eq:3}
\gamma(\widehat{\beta}_\St)  = \sum_{i = 1}^7 \lVert \varphi_i \rVert_* \lVert \psi_i \rVert_* \lVert W_i \rVert  = \sum_{i = 1}^7 \lVert U_i \rVert_\F \lVert V_i \rVert_\F \lVert W_i \rVert_\F \\
 = 12 + 2 \sqrt{2} \approx 14.83.
\end{equation}

\subsection{Bilinear stability of Winograd multiplication}

Winograd's algorithm \cite{HighamBook,winograd} computes a different set of intermediate quantities
\begin{alignat*}{2}
     M'_1 &= (A_{21}+A_{22}-A_{11})(B_{11}+B_{22}-B_{12}), \qquad &
     M'_5 &= (A_{21}+A_{22})(B_{12}-B_{11}), \\
     M'_2 &= A_{11}B_{11}, &
     M'_6 &= (A_{11}+A_{12}-A_{21}-A_{22})B_{22}, \\
     M'_3 &= A_{12}B_{21}, &
     M'_7 &= A_{22}(B_{11}+B_{22}-B_{12}-B_{21}), \\
     M'_4 &= (A_{11}-A_{21})(B_{22}-B_{12}), &
\end{alignat*}
and then compute the product via
\[
AB = 
    \begin{bmatrix}
    M'_2+M'_3 & M'_1+M'_2+M'_5+M'_6 \\
    M'_1+M'_2+M'_4-M'_7 & M'_1+M'_2+M'_4+M'_5
    \end{bmatrix}.
\]
Again this can be applied recursively.  Let $\widehat{\beta}_\W : \mathbb{R}^{2 \times 2} \times  \mathbb{R}^{2 \times 2} \to  \mathbb{R}^{2 \times 2} $ denote the Winograd's algorithm for $2 \times 2$ matrices. It is again routine to check that for $A, B \in \mathbb{R}^{2 \times 2}$,
\[
    \widehat{\beta}_\W(A,B) = \sum_{i = 1}^7 \varphi_i'(A)\psi_i'(B)W_i',
\]
where $\varphi_i'(A) = \tr(U_i^{\prime\tp}  A)$ and $\psi_i'(B) = \tr(V_i^{\prime\tp}  B)$ with
\begin{alignat*}{4}
    U_1' &= 
    \begin{bmatrix}
    -1 & 0 \\
    1 & 1
    \end{bmatrix},
    \quad
    & V_1' &= 
    \begin{bmatrix}
    1 & -1 \\
    0 & 1
    \end{bmatrix},
    \quad
    & W_1' &= 
    \begin{bmatrix}
    0 & 1 \\
    1 & 1
    \end{bmatrix}; \\
    U_2' &= 
    \begin{bmatrix}
    1 & 0 \\
    0 & 0
    \end{bmatrix},
    \quad
    & V_2' &= 
    \begin{bmatrix}
    1 & 0 \\
    0 & 0
    \end{bmatrix},
    \quad
    & W_2' &= 
    \begin{bmatrix}
    1 & 1 \\
    1 & 1
    \end{bmatrix};\\
    U_3' &= 
    \begin{bmatrix}
    0 & 1 \\
    0 & 0
    \end{bmatrix},
    \quad
    & V_3' &= 
    \begin{bmatrix}
    0 & 0 \\
    1 & 0
    \end{bmatrix},
    \quad
    & W_3' &= 
    \begin{bmatrix}
    1 & 0 \\
    0 & 0
    \end{bmatrix};\\
    U_4' &= 
    \begin{bmatrix}
    1 & 0 \\
    -1 & 0
    \end{bmatrix},
    \quad
    & V_4' &= 
    \begin{bmatrix}
    0 & -1 \\
    0 & 1
    \end{bmatrix},
    \quad
    & W_4' &= 
    \begin{bmatrix}
    0 & 0 \\
    1 & 1
    \end{bmatrix};\\
    U_5' &= 
    \begin{bmatrix}
    0 & 0 \\
    1 & 1
    \end{bmatrix},
    \quad
    & V_5' &= 
    \begin{bmatrix}
    -1 & 1 \\
    0 & 0
    \end{bmatrix},
    \quad
    & W_5' &= 
    \begin{bmatrix}
    0 & 1 \\
    0 & 1
    \end{bmatrix};\\
    U_6' &= 
    \begin{bmatrix}
    1 & 1 \\
    -1 & -1
    \end{bmatrix},
    \quad
    & V_6' &= 
    \begin{bmatrix}
    0 & 0 \\
    0 & 1
    \end{bmatrix},
    \quad
    & W_6' &= 
    \begin{bmatrix}
    0 & 1 \\
    0 & 0
    \end{bmatrix};\\
    U_7' &= 
    \begin{bmatrix}
    0 & 0 \\
    0 & 1
    \end{bmatrix},
    \quad
    & V_7' &= 
    \begin{bmatrix}
    1 & -1 \\
    -1 & 1
    \end{bmatrix},
    \quad
    & W_7' &= 
    \begin{bmatrix}
    0 & 0 \\
    -1 & 0
    \end{bmatrix}.
\end{alignat*}
With respect to the Frobenius norm, the growth factor of Winograd's algorithm is
\begin{equation}\label{winograd:eq:3}
\gamma(\widehat{\beta}_\W) = \sum_{i = 1}^7 \lVert \varphi_i' \rVert_* \lVert \psi_i' \rVert_* \lVert W_i' \rVert_\F 
 = \sum_{i = 1}^7 \lVert U_i' \rVert_\F \lVert V_i' \rVert_\F \lVert W_i' \rVert_\F \\
 = 7 + 4\sqrt{2} + 3 \sqrt{3} \approx 17.85.
\end{equation}

\subsection{Bilinear stability of conventional matrix multiplication}

For completeness we state the growth factor of the conventional algorithm for matrix multiplication $\widehat{\beta}_\C : \mathbb{R}^{2 \times 2} \times  \mathbb{R}^{2 \times 2} \to  \mathbb{R}^{2 \times 2} $,
\[
    \widehat{\beta}_\C(A,B) = \sum_{i,j,k = 1}^2 \tr(E_{ij}^\tp A)\tr(E_{jk}^\tp B) E_{ik},
\]
where $E_{ij} \in \mathbb{R}^{2 \times 2}$ denotes the standard basis matrix.
Its growth factor is easily seen to be
\[
\gamma(\widehat{\beta}_\C) = \sum_{i,j,k = 1}^2 \lVert E_{ij} \rVert_\F \lVert E_{jk} \rVert_\F \lVert E_{ik} \rVert_\F = 8.
\]
From \eqref{strassen:eq:3} and \eqref{winograd:eq:3}, we see that
\begin{equation}\label{strassen:winograd}
    \gamma(\widehat{\beta}_\W) > \gamma(\widehat{\beta}_\St) > \gamma(\widehat{\beta}_\C).
\end{equation}
The first inequality will be verified in the numerical experiments below; the second is consistent with the well-known fact \cite{HighamBook} that Strassen's algorithm is less stable than conventional multiplication.
In this case, the conventional algorithm attains the nuclear norm of two by two matrix multiplication, which has value $8$ \cite{Derksen}.

\subsection{Numerical experiments for fast matrix multiplications}

By Theorem~\ref{thm:main} and the sizes of the growth factors in \eqref{strassen:winograd}, we expect Strassen's algorithm to give more accurate results than Winograd's variant since it has a smaller growth factor. We test this statement with random matrices generated in three different ways: with (a) real entries drawn from the uniform distribution on $[-1,1]$, (b) real entries drawn from the standard normal distribution, (c) complex entries whose real and imaginary parts are drawn from the uniform distribution on $[-1,1]$. In the last case, note that our earlier discussions over $\mathbb{R}$ apply verbatim over $\mathbb{C}$ with the same growth factors. 
\begin{figure}[!hbt]
    \centering
    \begin{subfigure}[b]{0.49 \textwidth}
         \centering
         \includegraphics[width = \textwidth]{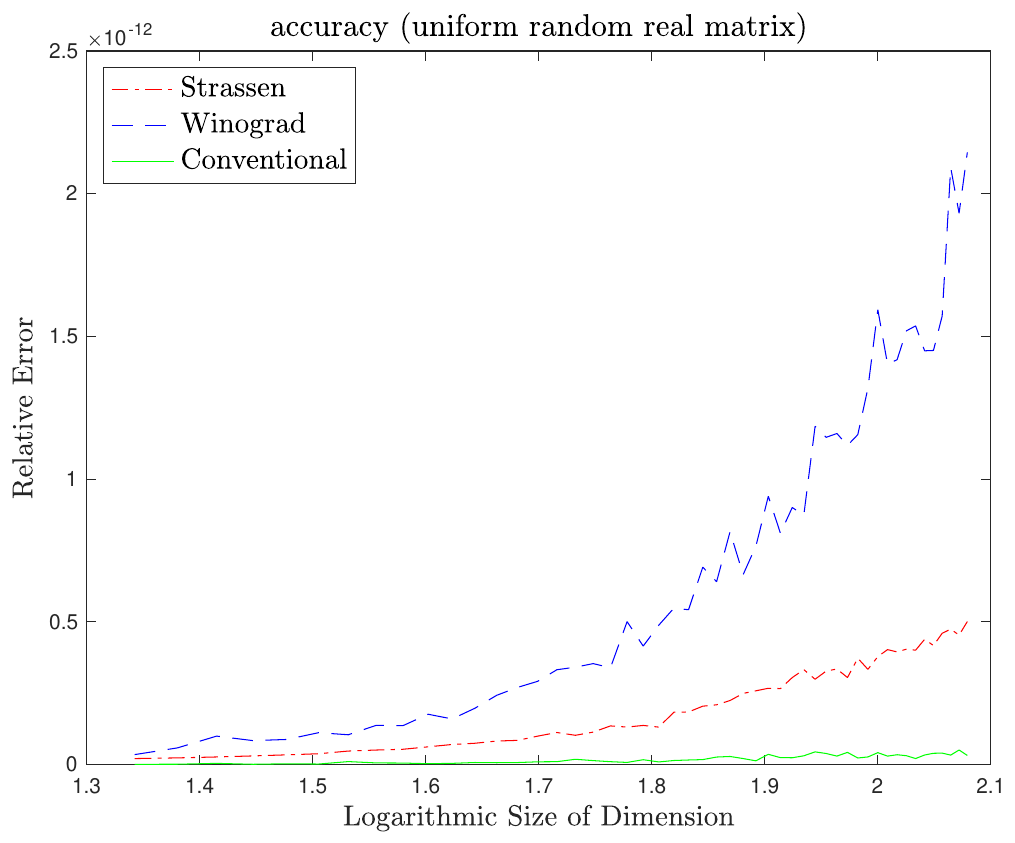}
         \caption{Real random matrices $\mathcal{U}[-1,1]$.}
    \end{subfigure}
    \hfill
    \begin{subfigure}[b]{0.49 \textwidth}
         \centering
         \includegraphics[width = \textwidth]{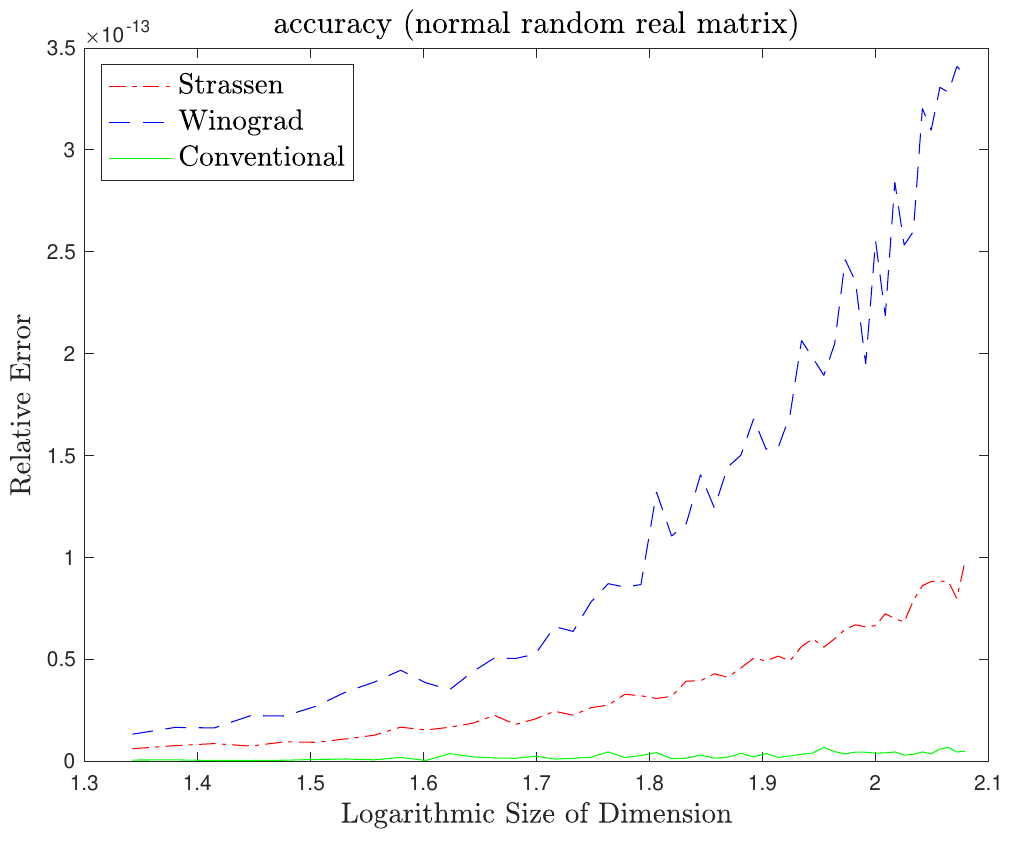}
         \caption{Real random matrices $\mathcal{N}(0,1)$.}
    \end{subfigure}
    \hfill
    \begin{subfigure}[b]{0.49 \textwidth}
         \centering
         \includegraphics[width = \textwidth]{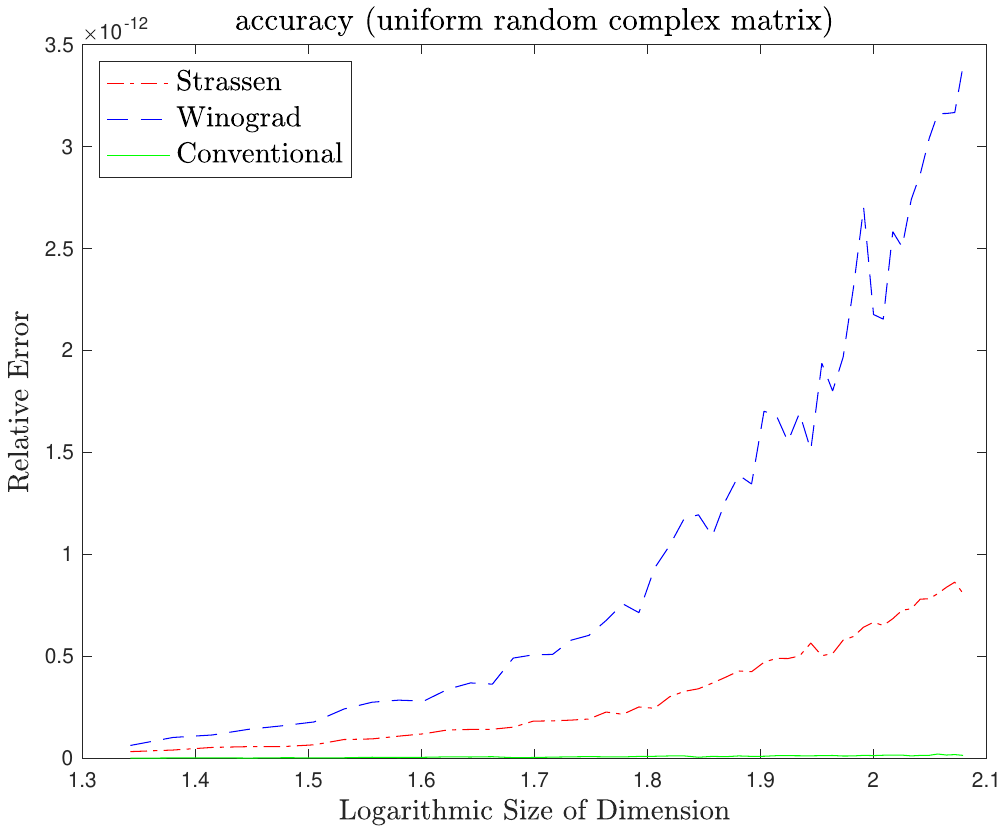}
         \caption{Complex random matrices $\mathcal{U}[-1,1] + \mathcal{U}[-1,1]i$.}
    \end{subfigure}
    \caption{Accuracy of Strassen's algorithm and Winograd's variant.}
    \label{fig:SW}
\end{figure}

In all cases, we compute $\widehat{\beta}_\St(A,B)$ and $\widehat{\beta}_\W(A,B)$ using Strassen's algorithm and Winograd's variant  respectively and compare the results against the exact value $\beta(A,B) = AB$ computed using the \textsc{Matlab} symbolic toolbox. From Figure~\ref{fig:SW}, we see that Strassen's algorithm is indeed more stable than Winograd's variant, substantiating Theorem~\ref{thm:main}. Even though the $14.83$ growth factor of Strassen's algorithm appears to differ only moderately from the $17.85$ growth factor of Winograd's variant, the effect is magnified multifold as a result of recursion --- these algorithms are applied recursively to an $n \times n$ matrix as a block $2 \times 2$ matrix $\lfloor \log_2 n \rfloor$ times. The conventional algorithm, which has a growth factor of $8$, is included in these plots for comparison.

\section{Complex multiplication}\label{sec:cplx}

As described towards the end of Section~\ref{sec:speed}, complex multiplication is an $\mathbb{R}$-bilinear operator $\beta_\mathbb{C} \in \mathbb{R}^2 \times \mathbb{R}^2 \to \mathbb{R}^2$  when we identify $\mathbb{C}\cong \mathbb{R}^2$, with the standard basis vectors in $\mathbb{R}^2$
\[
e_1 = \begin{bmatrix} 1 \\ 0 \end{bmatrix}, \qquad e_2 = \begin{bmatrix} 0 \\ 1 \end{bmatrix}
\]
corresponding to $1, i \in \mathbb{C}$. We write $e_1^*, e_2^* : \mathbb{R}^2 \to \mathbb{R}$ for the dual basis, i.e., linear functionals with
\[
e_1^*\biggl(\begin{bmatrix} a \\ b \end{bmatrix}\biggr) = a, \qquad
e_2^*\biggl(\begin{bmatrix} a \\ b \end{bmatrix}\biggr) = b.
\]

We will denote the regular algorithm $(a + bi)(c+ di) = (ac - bd) + i (bc+ad)$, Gauss's algorithm \eqref{eq:gauss}, and our new algorithm  \eqref{eq:ournew} by $\widehat{\beta}_\R, \widehat{\beta}_\G, \widehat{\beta}_\N$ respectively. For easy reference,
\begin{align*}
\widehat{\beta}_\R \biggl(\begin{bmatrix} a \\ b \end{bmatrix}, \begin{bmatrix} c \\ d \end{bmatrix}\biggr) &= 
\begin{bmatrix} ac - bd \\ bc+ad \end{bmatrix},\qquad \widehat{\beta}_\G \biggl(\begin{bmatrix} a \\ b \end{bmatrix}, \begin{bmatrix} c \\ d \end{bmatrix}\biggr) = \begin{bmatrix} ac - bd \\(a+b)(c+d)- ac -bd \end{bmatrix},\\[2ex]
\widehat{\beta}_\N \biggl(\begin{bmatrix} a \\ b \end{bmatrix}, \begin{bmatrix} c \\ d \end{bmatrix}\biggr) &= 
\begin{bmatrix}\frac{1}{2} \Bigl[
\Bigl(a + \frac{1}{\sqrt{3}}b\Bigr)
\Bigl(c + \frac{1}{\sqrt{3}}d\Bigr)
+\Bigl(a - \frac{1}{\sqrt{3}}b\Bigr)
\Bigl(c -  \frac{1}{\sqrt{3}}d\Bigr)
-\frac{8}{3}bd
\Bigr] \\[2ex]
\frac{i\sqrt{3}}{2} \Bigl[
\Bigl(a + \frac{1}{\sqrt{3}}b\Bigr)
\Bigl(c + \frac{1}{\sqrt{3}}d\Bigr)
-\Bigl(a - \frac{1}{\sqrt{3}}b\Bigr)
\Bigl(c - \frac{1}{\sqrt{3}}d\Bigr)
\Bigr]
\end{bmatrix}.
\end{align*}
They correspond to the decompositions
\begin{align}
\widehat{\beta}_\R &= (e_1^* \otimes e_1^* -e_2^* \otimes  e_2^* ) \otimes  e_1 +  (e_1^* \otimes  e_2^* + e_2^* \otimes   e_1^* ) \otimes e_2, \label{eq:standarddecomp} \\
\widehat{\beta}_\G &= (e_1^*+e_2^*) \otimes  (e_1^*+e_2^*) \otimes  e_2+ e_1^* \otimes e_1^* \otimes (e_1-e_2) - e_2^* \otimes  e_2^* \otimes (e_1+e_2), \label{eq:Gaussdecomp}\\
\widehat{\beta}_\N &= \frac{4}{3}\biggl(\biggl[\frac{\sqrt{3}}{2}e_1^*  +\frac{1}{2}e_2^*  \biggr] \otimes
\biggl[\frac{\sqrt{3}}{2}e_1^*  +\frac{1}{2}e_2^*  \biggr] \otimes \biggl[\frac{1}{2}e_1+\frac{\sqrt{3}}{2}e_2\biggr] \nonumber \\
&\qquad+ \biggl[\frac{\sqrt{3}}{2}e_1^*  -\frac{1}{2}e_2^*  \biggr] \otimes \biggl[\frac{\sqrt{3}}{2}e_1^*  -\frac{1}{2}e_2^*  \biggr] \otimes \biggl[\frac{1}{2}e_1-\frac{\sqrt{3}}{2}e_2\biggr]
-e_2^* \otimes e_2^* \otimes  e_1\biggr). \label{eq:FLdecomp}
\end{align}

\subsection{Bilinear stability of complex multiplication algorithms}

Recall from Section~\ref{sec:speed} that $\rank(\beta_\mathbb{C}) = 3 = \brank(\beta_\mathbb{C})$, i.e., both Gauss's algorithm and our new algorithm have optimal bilinear complexity whether in the exact or approximate sense. One may also show that $\beta_\mathbb{C}$ has nuclear norm \cite[Lemma~6.1]{nuclear} is given by
\[
\lVert \beta_\mathbb{C}\rVert_{\nu} = 4.
\]
The growth factor of the regular algorithm \eqref{eq:standarddecomp} attains this minimum value,
\begin{align*}
\gamma(\widehat{\beta}_\R) &= \lVert e_1^*\rVert_*\lVert e_1^*\rVert_*\lVert  e_1\rVert  + \lVert - e_2^* \rVert_*\lVert  e_2^*\rVert_*\lVert  e_1 \rVert + \lVert e_1^*\rVert_*\lVert  e_2^*\rVert_*\lVert  e_2\rVert  +\lVert e_2^* \rVert_*\lVert  e_1^*\rVert_*\lVert  e_2\rVert \\
 &= 4 = \lVert \beta_\mathbb{C}\rVert_{\nu},
\end{align*}
as does our new algorithm \eqref{eq:FLdecomp},
\begin{align*}
\gamma(\widehat{\beta}_\N) &=\frac{4}{3}\biggl(\biggl\lVert \frac{\sqrt{3}}{2}e_1^*  +\frac{1}{2}e_2^*  \biggr\rVert_*
\biggl\lVert \frac{\sqrt{3}}{2}e_1^*  +\frac{1}{2}e_2^*  \biggr\rVert_* \biggl\lVert \frac{1}{2}e_1+\frac{\sqrt{3}}{2}e_2\biggr\rVert \\
&\qquad+ \biggl\lVert \frac{\sqrt{3}}{2}e_1^*  -\frac{1}{2}e_2^*  \biggr\rVert_*  \biggl\lVert \frac{\sqrt{3}}{2}e_1^*  -\frac{1}{2}e_2^*  \biggr\rVert_*   \biggl\lVert \frac{1}{2}e_1-\frac{\sqrt{3}}{2}e_2\biggr\rVert 
+ \lVert e_2^*\rVert_*\lVert   e_2^* \rVert_*\lVert  e_1 \rVert \biggr)\\
&= 4 = \lVert \beta_\mathbb{C}\rVert_{\nu},
\end{align*}
but not  Gauss's algorithm \eqref{eq:Gaussdecomp},
\begin{align*}
\gamma(\widehat{\beta}_\G) &= \lVert e_1^*+e_2^*\rVert_*\lVert  e_1^*+e_2^*\rVert_*\lVert  e_2 \rVert +\lVert e_1^*\rVert_*\lVert  e_1^*\rVert_*\lVert  e_1-e_2\rVert
 + \lVert - e_2^*\rVert_*\lVert  e_2^*\rVert_*\lVert  e_1+e_2\rVert \\
&=2(1+ \sqrt{2}) > \lVert \beta_\mathbb{C}\rVert_{\nu}.
\end{align*}
So Gauss's algorithm $\widehat{\beta}_\G$ is faster (by bilinear complexity) but less stable (by bilinear stability) than the regular algorithm. Our new algorithm $\widehat{\beta}_\N$ on the other hand is optimal in both measures, attaining both $\rank(\beta_\mathbb{C})$ and $\lVert \beta_\mathbb{C}\rVert_{\nu}$.

We stress that numerical stability is too complicated an issue to be completely covered by the simple framework of bilinear stability. For instance, from the perspective of cancellation errors, our new algorithm also suffers from the issue pointed out in \cite[Section~23.2.4]{HighamBook} for Gauss's algorithm. By choosing $z=w$ and $b =\sqrt{3}/a$, our algorithm \eqref{eq:FLdecomp}  computes
\[
\frac{1}{2} \Bigl[
\Bigl(a + \frac{1}{a}\Bigr)^2
+\Bigl(a - \frac{1}{a}\Bigr)^2
-\frac{8}{a^2}
\Bigr] + \frac{i\sqrt{3}}{2} \Bigl[
\Bigl(a + \frac{1}{a}\Bigr)^2
-\Bigl(a - \frac{1}{a}\Bigr)^2
\Bigr]\eqqcolon x + iy.
\]
There will be cancellation error in the computed real part $\widehat{x}$ when $\lvert a \rvert$ is small and likewise in the computed imaginary part $\widehat{y}$ when $\lvert a \rvert$ is large. Nevertheless, as discussed in \cite[Section~23.2.4]{HighamBook}, the new algorithm \eqref{eq:FLdecomp} is still stable in the weaker sense of having acceptably small $\lvert x - \widehat{x} \rvert/\lvert z \rvert$ and $\lvert y - \widehat{y} \rvert/\lvert z \rvert$ even if $\lvert x - \widehat{x} \rvert/\lvert x \rvert$ or $\lvert y - \widehat{y} \rvert/\lvert y \rvert$ might be large.

\subsection{Error analysis of new algorithm applied to matrices}\label{sec:roundcplx}

While using Gauss's algorithm or our new algorithm for multiplying of complex \emph{numbers} is a pointless overkill, they become useful when applied to the multiplication of complex \emph{matrices}. Note that any complex matrices $A + iB, C + i D \in \mathbb{C}^{n \times n}$ may be multiplied via their real and imaginary parts $A,B,C,D \in \mathbb{R}^{n \times n}$:
\begin{equation}\label{eq:Regular}
(A + iB)(C+ iD)= (AC - BD) + i[AD + BC],
\end{equation}
allowing us to focus our attention on designing algorithms for real matrix products. In this regard, Gauss's algorithm applied in the form
\begin{equation}\label{eq:Gauss}
(A + iB)(C+ iD)= (AC - BD) + i[(A+B)(C+D) -AC -BD]
\end{equation}
reduces the number of real matrix products from four to three at the expense of more matrix additions. This represents an enormous saving as matrix products are invariably much more expensive than matrix additions.
Our new algorithm \eqref{eq:ournew} likewise applies in the form
\begin{align}
(A + iB)(C+ iD) &= \frac{1}{2} \biggl[
\biggl(A + \frac{1}{\sqrt{3}}B\biggr)
\biggl(C + \frac{1}{\sqrt{3}}D\biggr)
+\biggl(A - \frac{1}{\sqrt{3}}B\biggr)
\biggl(C -  \frac{1}{\sqrt{3}}D\biggr)
-\frac{8}{3}BD
\biggr] \nonumber\\
&\qquad+
\frac{i\sqrt{3}}{2} \biggl[
\biggl(A + \frac{1}{\sqrt{3}}B\biggr)
\biggl(C + \frac{1}{\sqrt{3}}D\biggr)
-\biggl(A - \frac{1}{\sqrt{3}}B\biggr)
\biggl(C - \frac{1}{\sqrt{3}}D\biggr)
\biggr], \label{eq:New}
\end{align}
trading expensive matrix products for inexpensive scalar multiplications and additions.

The following is an error analysis of \eqref{eq:New}, i.e., our new algorithm applied to complex matrix multiplication. We emulate a similar analysis for Gauss's algorithm in \cite{Higham,HighamBook}, assuming in particular that the real matrix multiplications involved are performed using the conventional algorithm (as opposed to Strassen's or Winograd's). We remind the reader that conventional matrix multiplication has the simple error bound
\begin{equation} \label{eq:smm}
    |AB - \fl(AB)| \le n|A||B|\boldsymbol{\mathsf{u}} + O(\boldsymbol{\mathsf{u}}^2)
\end{equation}
for $A,B \in \mathbb{R}^{n \times n}$.
\begin{theorem}[Error analysis for our new algorithm]\label{thm:new}
Let $(A+i B)(C+i D) = F + i G$ with $F, G \in \mathbb{R}^{n \times n}$ and let $\widehat{F}_\N, \widehat{G}_\N$ be computed via \eqref{eq:New} in floating point arithmetic satisfying \eqref{eq:float}. 
Then
\begin{align}
|F - \widehat{F}_\N| &\le (n+7) \biggl(|A|+\frac{1}{\sqrt{3}}|B|\biggr)\biggl(|C|+\frac{1}{\sqrt{3}}|D|\biggr)\boldsymbol{\mathsf{u}} + \biggl(\frac{4}{3}n + 4\biggr)|B||D|\boldsymbol{\mathsf{u}} + O(\boldsymbol{\mathsf{u}}^2), \label{Gun:New1}\\
|G - \widehat{G}_\N| &\le \sqrt{3}(n+6)\biggl(|A|+\frac{1}{\sqrt{3}}|B|\biggr)\biggl(|C|+\frac{1}{\sqrt{3}}|D|\biggr)\boldsymbol{\mathsf{u}} + O(\boldsymbol{\mathsf{u}}^2), \label{Gun:New2}
\end{align}
where the inequality $\le$ and absolute value $\lvert\, \cdot\, \rvert$ both apply in a coordinatewise sense.
\end{theorem}
\begin{proof}
Following \cite{HighamBook}, we use the same letter $\delta$ to denote the error incurred in each step of our algorithm. So, for example,
\[
   \fl(B/\sqrt{3}) = B/\sqrt{3} + \delta B/\sqrt{3}.
\]
In the following we will define matrices $H_i$ and let $\widehat{H}_i$ be its computed value, $i =1,\dots,8$.

Let $H_1 \coloneqq A+B/\sqrt{3}$. Then
\begin{align*}
    \widehat{H}_1 &= \fl(A + B/\sqrt{3} + \delta B/\sqrt{3}) = (A + B/\sqrt{3} + \delta B/\sqrt{3}) (1 + \delta) \\
    &= A+B/\sqrt{3} + \delta (A + 2B/\sqrt{3}) + O(\boldsymbol{\mathsf{u}}^2) \\
    &= H_1 + 2 \Delta_1 + O(\boldsymbol{\mathsf{u}}^2), \qquad  |\Delta_1| \le  (|A| + |B|/\sqrt{3})\boldsymbol{\mathsf{u}}.
\end{align*}
Similarly $H_2 \coloneqq  C+D/\sqrt{3}$ satisfies
\[
    \widehat{H}_2 = H_2 + 2 \Delta_2 + O(\boldsymbol{\mathsf{u}}^2), \qquad |\Delta_2| \le (|C| + |D|/\sqrt{3})\boldsymbol{\mathsf{u}}.
\]

Let $H_3 \coloneqq  (A+B/\sqrt{3})(C+D/\sqrt{3})$.
By \eqref{eq:smm},
\begin{equation} \label{eq:stable:01}
    \widehat{H}_3 =  (A+B/\sqrt{3} + 2 \Delta_1)(C+D/\sqrt{3} + 2 \Delta_2) + n \Delta_3 + O(\boldsymbol{\mathsf{u}}^2)
\end{equation}
where
\begin{equation} \label{eq:stable:02}
\begin{aligned}
    |\Delta_3| & \le  |(A+B/\sqrt{3} + 2 \Delta_1)||(C+D/\sqrt{3} + 2 \Delta_2)|\boldsymbol{\mathsf{u}} \\
    & \le (|A|+|B|/\sqrt{3} + 2 |\Delta_1|)(|C|+|D|/\sqrt{3} +2|\Delta_2|)\boldsymbol{\mathsf{u}} \\
    & \le (|A|+|B|/\sqrt{3} + 2 \boldsymbol{\mathsf{u}} (|A| + |B|/\sqrt{3}))(|C|+|D|/\sqrt{3} +2\boldsymbol{\mathsf{u}} (|C| + |D|/\sqrt{3}))\boldsymbol{\mathsf{u}} \\
    & \le (|A|+|B|/\sqrt{3}) (|C|+|D|/\sqrt{3})\boldsymbol{\mathsf{u}} + O(\boldsymbol{\mathsf{u}}^2).
\end{aligned}
\end{equation}
By \eqref{eq:stable:01} and \eqref{eq:stable:02},
\begin{equation} \label{eq:stable:03}
\begin{aligned}
    \widehat{H}_3 & = (A+B/\sqrt{3})(C+D/\sqrt{3}) + 2\Delta_1(C+D/\sqrt{3}) \\
&\qquad + 2(A+B/\sqrt{3})\Delta_2 + n \Delta_3 + O(\boldsymbol{\mathsf{u}}^2) \\
    & = H_3 + (n+4)\Delta_4 + O(\boldsymbol{\mathsf{u}}^2)
\end{aligned}
\end{equation}
where
\[
    |\Delta_4| \le  (|A|+|B|/\sqrt{3}) (|C|+|D|/\sqrt{3})\boldsymbol{\mathsf{u}}.
\]
Similarly $H_4 \coloneqq  (A-B/\sqrt{3})(C-D/\sqrt{3})$ satisfies
\begin{equation} \label{eq:stable:04}
    \widehat{H}_4 = H_4 + (n+4)\Delta_5 + O(\boldsymbol{\mathsf{u}}^2)
\end{equation}
where
\[
    |\Delta_5| \le  (|A|+|B|/\sqrt{3}) (|C|+|D|/\sqrt{3})\boldsymbol{\mathsf{u}}.
\]

Let $H_5 \coloneqq  (A+B/\sqrt{3})(C+D/\sqrt{3}) + (A-B/\sqrt{3})(C-D/\sqrt{3})$.
By \eqref{eq:stable:03} and \eqref{eq:stable:04},
\begin{equation} \label{eq:stable:05}
\begin{aligned}
    \widehat{H}_5
    & = [H_3 + (n+4)\Delta_4 + H_4 + (n+4)\Delta_5] (1+\delta) + O(u^2) \\
    & = H_5 + (2n+10)\Delta_6 + O(\boldsymbol{\mathsf{u}}^2)
\end{aligned}
\end{equation}
where
\[
    |\Delta_6| \le \boldsymbol{\mathsf{u}} (|A|+|B|/\sqrt{3}) (|C|+|D|/\sqrt{3}).
\]

Let $H_6 \coloneqq  8/3 BD$. Then
\begin{equation} \label{eq:stable:06}
\begin{aligned}
    \widehat{H}_6 & = \fl(8/3 (BD + n \Delta_7)) + O(\boldsymbol{\mathsf{u}}^2) \\
    & = 8/3(BD + n \Delta_7) (1 + \delta) +O(\boldsymbol{\mathsf{u}}^2) \\
    & = H_6 + 8/3 (n+1) \Delta_8 + O(\boldsymbol{\mathsf{u}}^2)
\end{aligned}
\end{equation}
where 
\[
    |\Delta_7| \le |B||D|\boldsymbol{\mathsf{u}}, \qquad  |\Delta_8| \le |B||D|\boldsymbol{\mathsf{u}}.
\]

Let $H_7 \coloneqq  H_5 - H_6$. 
By \eqref{eq:stable:05} and \eqref{eq:stable:06},
\begin{align*}
    \widehat{H}_7
    & = [H_5 + (2n+10)\Delta_6
    - H_6 - 8/3 (n+1) \Delta_8] (1 + \delta) + O(\boldsymbol{\mathsf{u}}^2) \\
    & = H_7 + (2n+12) \Delta_9 + 8/3 (n+2) \Delta_{10} + O(\boldsymbol{\mathsf{u}}^2)
\end{align*}
where
\[
    |\Delta_9| \le  (|A|+|B|/\sqrt{3}) (|C|+|D|/\sqrt{3})\boldsymbol{\mathsf{u}}, \qquad  |\Delta_{10}| \le |B||D|\boldsymbol{\mathsf{u}}.
\]
Then
\begin{align*}
    \widehat{F}_\N & = (1+\delta)[H_7
     + (2n+12) \Delta_9 + 8/3 (n+2) \Delta_{10}]/2 + O(\boldsymbol{\mathsf{u}}^2) \\
    & = F + (n+7) \Delta_{11} + 4/3(n+3) \Delta_{12} + O(\boldsymbol{\mathsf{u}}^2)
\end{align*}
where
\[
    |\Delta_{11}| \le  (|A|+|B|/\sqrt{3}) (|C|+|D|/\sqrt{3}) \boldsymbol{\mathsf{u}}, \qquad  |\Delta_{12}| \le |B||D|\boldsymbol{\mathsf{u}},
\]
and from which we obtain \eqref{Gun:New1}.

Let $H_8 \coloneqq  (A+B/\sqrt{3})(C+D/\sqrt{3}) - (A-B/\sqrt{3})(C-D/\sqrt{3})$.
Similar to \eqref{eq:stable:05}, we have
\[
    \widehat{H}_8 = H_8 - (2n+10)\Delta_{13} + O(\boldsymbol{\mathsf{u}}^2)
\]
where
\[
    |\Delta_{13}| \le  (|A|+|B|/\sqrt{3}) (|C|+|D|/\sqrt{3})\boldsymbol{\mathsf{u}}.
\]
Then
\begin{align*}
    \widehat{G}_\N & = \sqrt{3}/2 [H_8 - (2n+10)\Delta_{13}] (1+\delta) + O(\boldsymbol{\mathsf{u}}^2) \\
    & = G + \sqrt{3}(n+6) \Delta_{14} + O(\boldsymbol{\mathsf{u}}^2)
\end{align*}
where
\[
    |\Delta_{14}| \le  (|A|+|B|/\sqrt{3}) (|C|+|D|/\sqrt{3})\boldsymbol{\mathsf{u}},
\]
from which we obtain \eqref{Gun:New2}.
\end{proof}

If we compute the matrices $F,G$ in Theorem~\ref{thm:new} using Gauss's algorithm \eqref{eq:Gauss} with floating point arithmetic and let the results be  $\widehat{F}_\G$ and $\widehat{G}_\G$, then the corresponding error bounds \cite{Higham,HighamBook} are
\begin{equation} \label{Gun:Gauss}
\begin{aligned}
    |F - \widehat{F}_\G| &\le (n + 1)( |A||C| + |B||D|)\boldsymbol{\mathsf{u}} + O(\boldsymbol{\mathsf{u}}^2),\\
    |G - \widehat{G}_\G| &\le (n + 4)[(|A| + |B|) (|C|+|D|) + |A||C| + |B||D|]\boldsymbol{\mathsf{u}} + O(\boldsymbol{\mathsf{u}}^2).
\end{aligned}
\end{equation}
When $n \to \infty$, we have $n+c \approx n$ for any constant $c$. Hence the errors in \eqref{Gun:New1} and \eqref{Gun:New2}  are dominated by
\begin{align*}
|F - \widehat{F}_\N| &\sim n \biggl[|A||C| + \frac{5}{3} |B||D| + \frac{1}{\sqrt{3}} |A||D| + \frac{1}{\sqrt{3}}|B||C|\biggr]\boldsymbol{\mathsf{u}},\\
|G - \widehat{G}_\N| &\sim n \biggl[\sqrt{3}|A||C| + |B||C| + |A||D| + \frac{1}{\sqrt{3}} |B||D| \biggr] \boldsymbol{\mathsf{u}},
\end{align*}
whereas those in \eqref{Gun:Gauss} are dominated by
\begin{align*}
|F - \widehat{F}_\G| &\sim  n(|A||C| + |B||D| )\boldsymbol{\mathsf{u}}, \\
 |G - \widehat{G}_\G| &\sim  n(2|A||C| + 2|B||D| +|A||D| + |B||C|)\boldsymbol{\mathsf{u}} .
\end{align*}
For easy comparison suppose the magnitudes of the entries in $A,B,C,D$  are all approximately $\theta$, then these reduce to
\begin{equation}\label{eq:compare}
\begin{alignedat}{2}
|F - \widehat{F}_\N| &\sim 3.8 n^2 \theta^2,\qquad & |G - \widehat{G}_\N|  &\sim  4.3 n^2 \theta^2,\\
|F - \widehat{F}_\G| &\sim  2n^2 \theta^2, & |G - \widehat{G}_\G|  &\sim  6 n^2 \theta^2.
\end{alignedat}
\end{equation}
So Gauss's algorithm gives an imaginary part that is three times less accurate than its real part. Note the the imaginary part of Gauss's algorithm accounts for all its computational savings; the real part is just the regular algorithm.  On the other hand, our algorithm balances the accuracy of both the real and imaginary parts by spreading out the computational savings across both parts.

To quantify this, we use the \emph{max norm}. For a complex matrix $A + i B \in \mathbb{C}^{n \times n}$, this is
\begin{equation}\label{eq:maxnorm}
    \lVert A + i B \rVert_{\max} \coloneqq \max \{\lvert a_{ij} \rvert, \lvert b_{ij} \rvert : i,j = 1,\dots,n\}.
\end{equation}
The max norm differs from the usual matrix $\infty$-norm given by maximum row sum used in \cite{Higham,HighamBook}. We favor the max norm as it is the strictest measure of numerical accuracy --- a small max norm error implies that each entry is accurate as opposed to accurate on average.

If we denote the matrices resulting from Gauss's algorithm and our new algorithm by
\[
 \widehat{E}_\G \coloneqq \widehat{F}_\G + i  \widehat{G}_\G, \qquad  \widehat{E}_\N \coloneqq  \widehat{F}_\N + i  \widehat{G}_\N
\]
respectively, we expect $\lVert E - \widehat{E}_\N\rVert_{\max}$ to be smaller than $\lVert E - \widehat{E}_\G\rVert_{\max}$. The extensive experiments in Section~\ref{sec:expr} will attest to this.

\subsection{Derivation of our algorithm}

It is perhaps instructive to include a description of how one may derive the algorithm in \eqref{eq:ournew} by minimizing growth factor. Observe that Gauss's algorithm \eqref{eq:gauss} includes the term $(a+b)(c+d)$, which adds $2$ to its growth factor. We seek to reduce the growth factor by replacing it with $(a+rb)(c+rd)$ for some  shrinkage $r \in (0,1)$, which leads to a family of algorithms parameterized by $r$:
\begin{multline*}
(a + bi)(c+ di) = \frac{1}{2} [
(a + rb)
(c + rd)
+(a - rb)
(c -  rd)
-(2r^2+2)bd
]
\\
+
\frac{i}{2r} [
(a + rb)
(c + rd)
-(a - rb)
(c - rd)
].
\end{multline*}
Let $g(r)$ denote the growth factor. A simple calculation shows that
\[
    g(r) = \frac{1}{r} (1+r^2)^{3/2} + r^2 + 1,
\]
which has a minimum of $4$ attained at $r = 1/\sqrt{3}$, giving us \eqref{eq:ournew}. Note that \eqref{eq:ournew} is not unique; another algorithm with growth factor $4$ is given by
\begin{multline*}
(a + bi)(c+ di) =
\frac{\sqrt{3}}{2} \biggl[
\biggl(a + \frac{1}{\sqrt{3}}b\biggr)
\biggl(\frac{1}{\sqrt{3}}c - d\biggr)
+\biggl(a - \frac{1}{\sqrt{3}}b\biggr)
\biggl(\frac{1}{\sqrt{3}}c + d\biggr)
\biggr]
\\
+
 \frac{i}{2} \biggl[
\biggl(a - \frac{1}{\sqrt{3}}b\biggr)
\biggl(\frac{1}{\sqrt{3}}c + d\biggr)
-\biggl(a + \frac{1}{\sqrt{3}}b\biggr)
\biggl(\frac{1}{\sqrt{3}}c - d\biggr)
+\frac{8}{3}bc
\biggr],
\end{multline*}
which may be obtained from  \eqref{eq:ournew} by substituting $c = di$ and $d = -ci$.

\section{Experiments for new complex matrix multiplication algorithm}\label{sec:expr}

The goal of this section is to provide numerical evidence to show that our new algorithm \eqref{eq:New} for  complex matrix multiplication is 
\begin{itemize}
    \item nearly as stable as the regular algorithm \eqref{eq:Regular}, and
    \item nearly as fast as Gauss's algorithm \eqref{eq:Gauss}.
\end{itemize}
We begin with routine experiments comparing the three algorithms \eqref{eq:Regular}, \eqref{eq:Gauss}, \eqref{eq:New} on random matrices, and move on to three actual applications: matrix polynomial evaluations, unitary transformations, and the increasingly popular complex-valued neural networks. The results, we think, show that our new algorithm can be a realistic replacment for Gauss's algorithm in engineering applications.

\subsection{Speed of the algorithms}

We generate random $A+i B, C + i D \in \mathbb{C}^{n \times n}$ with entries of $A,B,C,D$ drawn uniformly in $[-1,1]$; the results with standard normal are similar and omitted. We increase $n$ from $2100$ to $7000$ in steps of $100$. The product $(A+i B)(C+i D)$ is computed numerically with the regular algorithm \eqref{eq:Regular}, Gauss's algorithm \eqref{eq:Gauss}, and our new algorithm \eqref{eq:New}. For each $n$, we generate ten different matrices and record the average time taken for each algorithm and plot these in Figure~\ref{fig:Speed}, with wall time (in seconds) for vertical axis and $\log_{10}(n)$ for horizontal axis. The time taken by \textsc{Matlab}'s internal function for complex matrix multiplication is virtually indistinguishable from that of the regular algorithm and therefore omitted.

Consistent with the predictions of bilinear complexity, our new algorithm has roughly the same computation time as Gauss's algorithm, at roughly $3/4$ the time taken by the regular algorithm. We will perform more speed experiments in conjunction with our accuracy experiments in Section~\ref{exp:accuracy:general}.
\begin{figure}[ht]
    \centering
    \includegraphics[scale = 0.5]{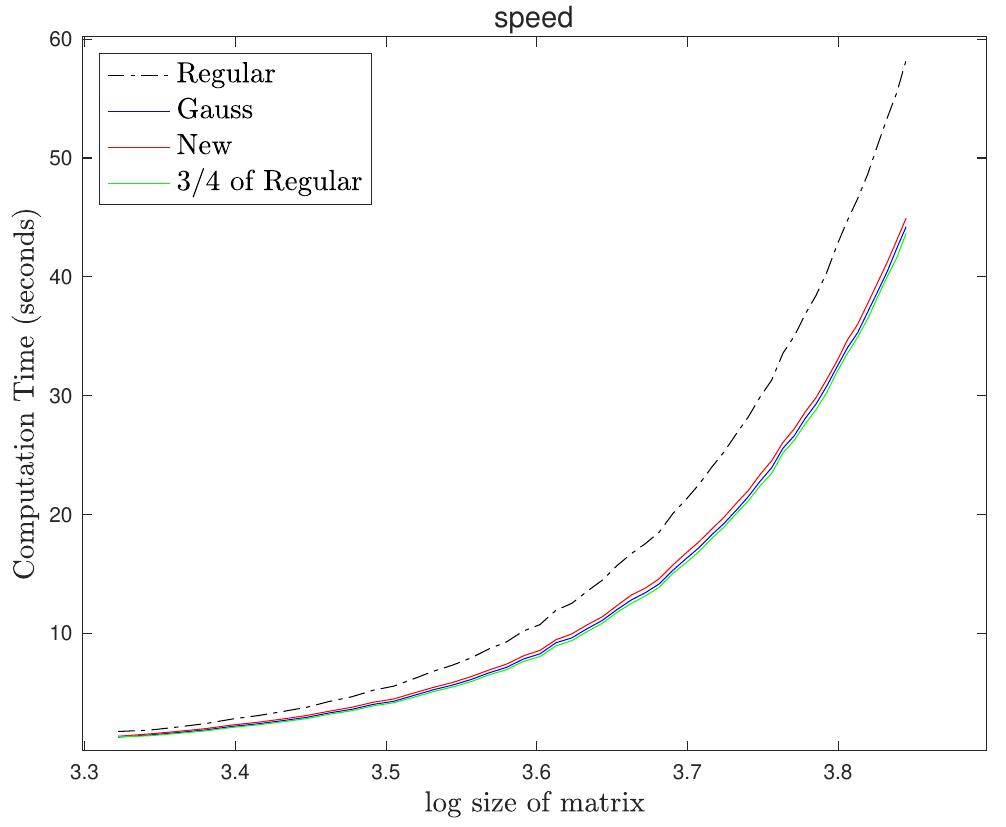}
    \caption{Speed of the three algorithms for complex matrix multiplication.}
    \label{fig:Speed}
\end{figure}

\subsection{Accuracy of the algorithms} \label{exp:accuracy:general}

\begin{figure}[ht]
    \centering
	 \includegraphics[width = 0.49\textwidth]{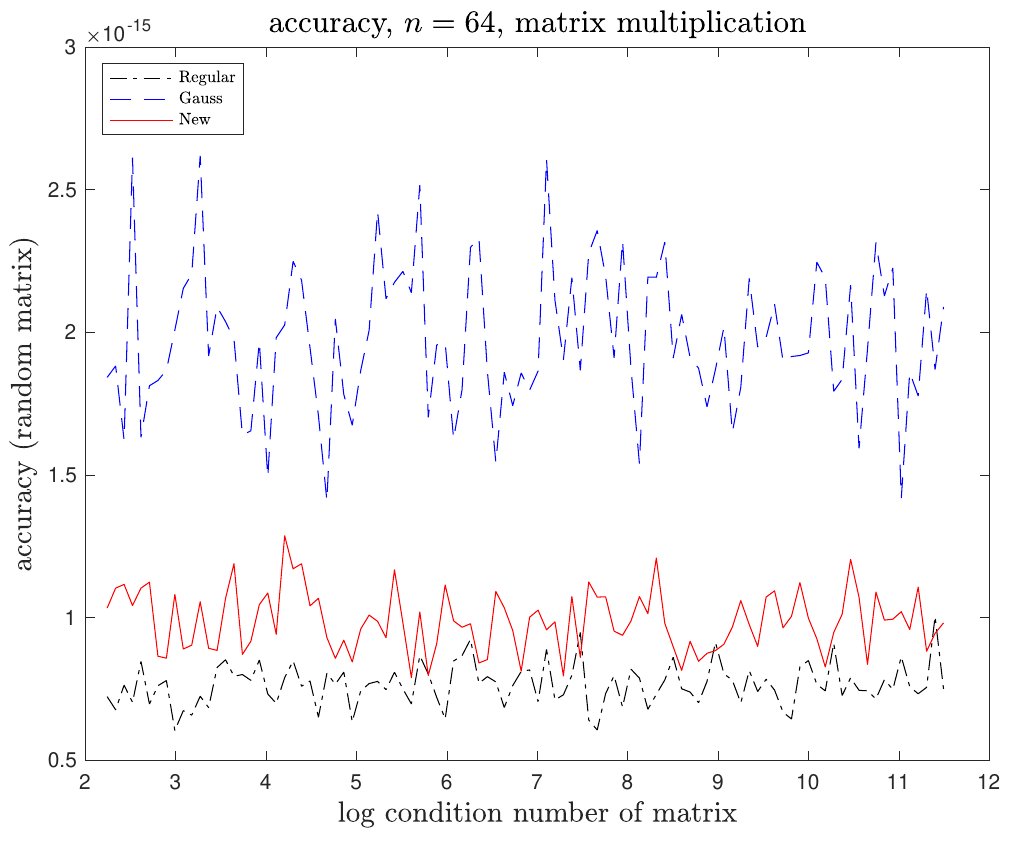}
         \includegraphics[width = 0.49\textwidth]{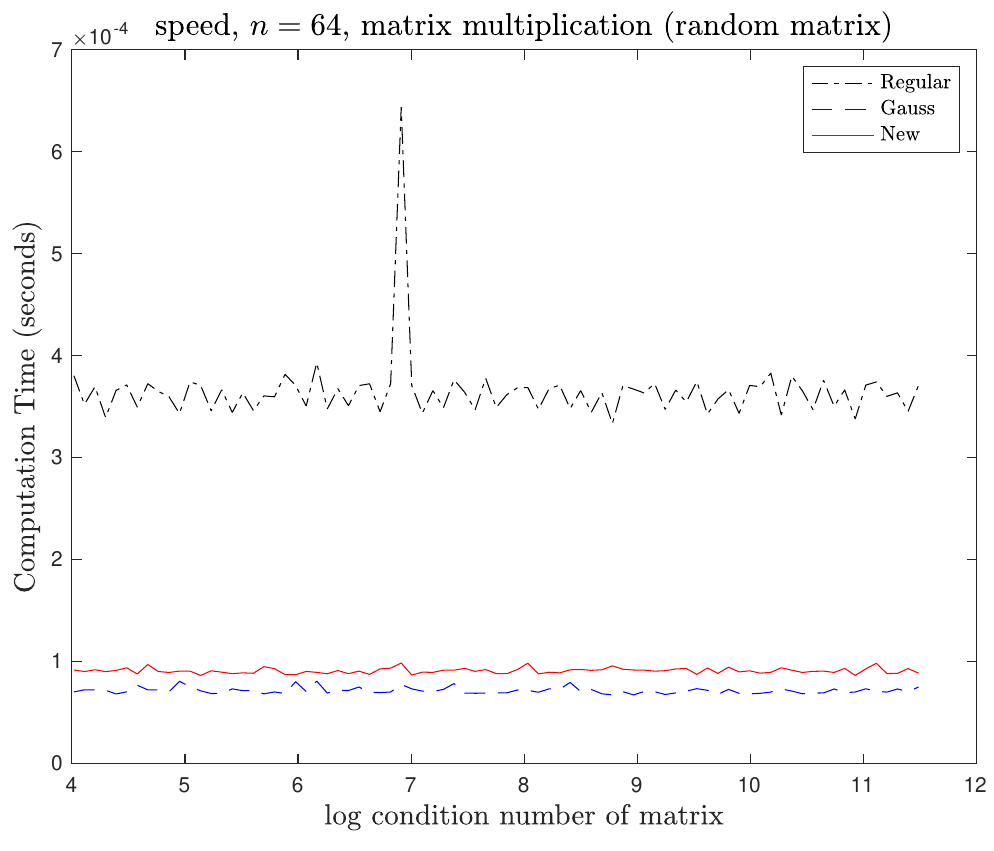}
         \includegraphics[width = 0.49\textwidth]{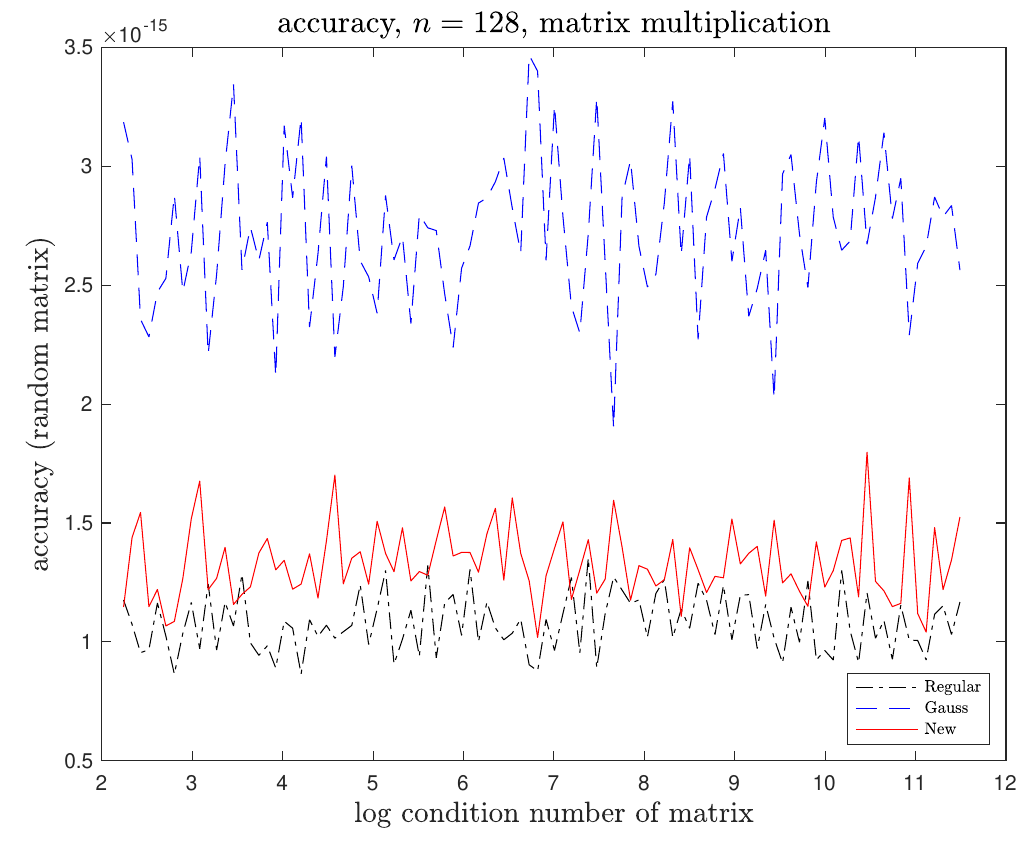}
         \includegraphics[width = 0.49\textwidth]{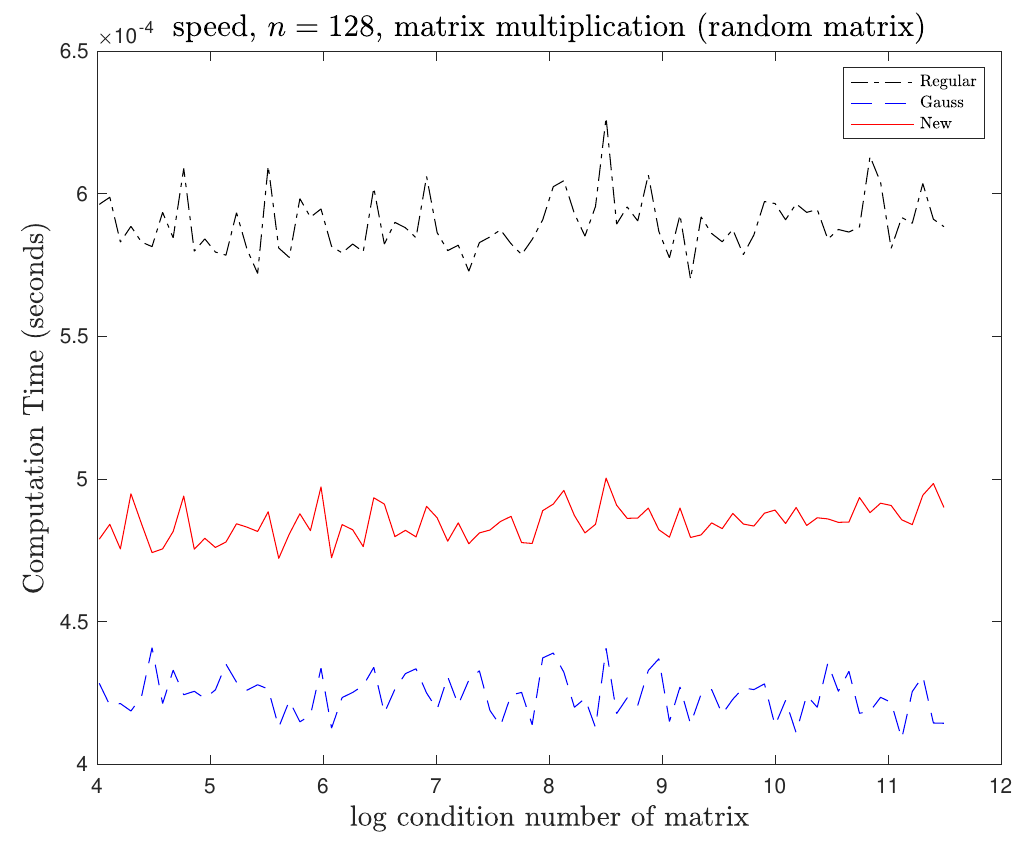}
         \includegraphics[width = 0.49\textwidth]{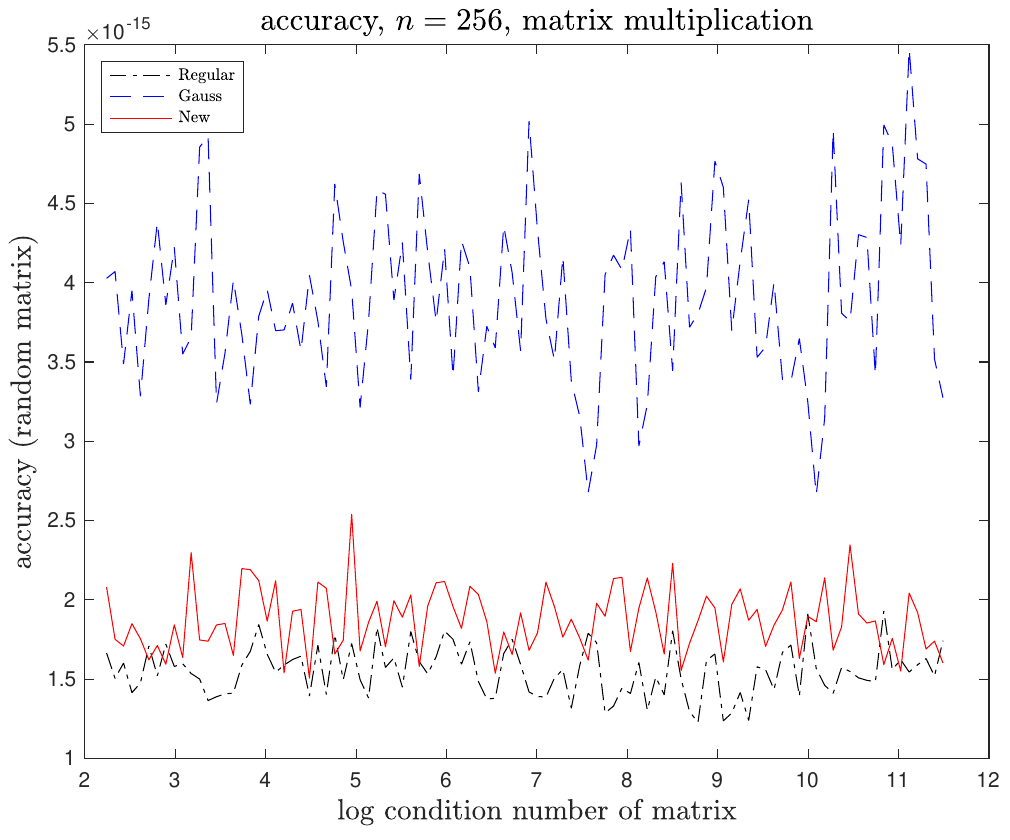}
         \includegraphics[width = 0.49\textwidth]{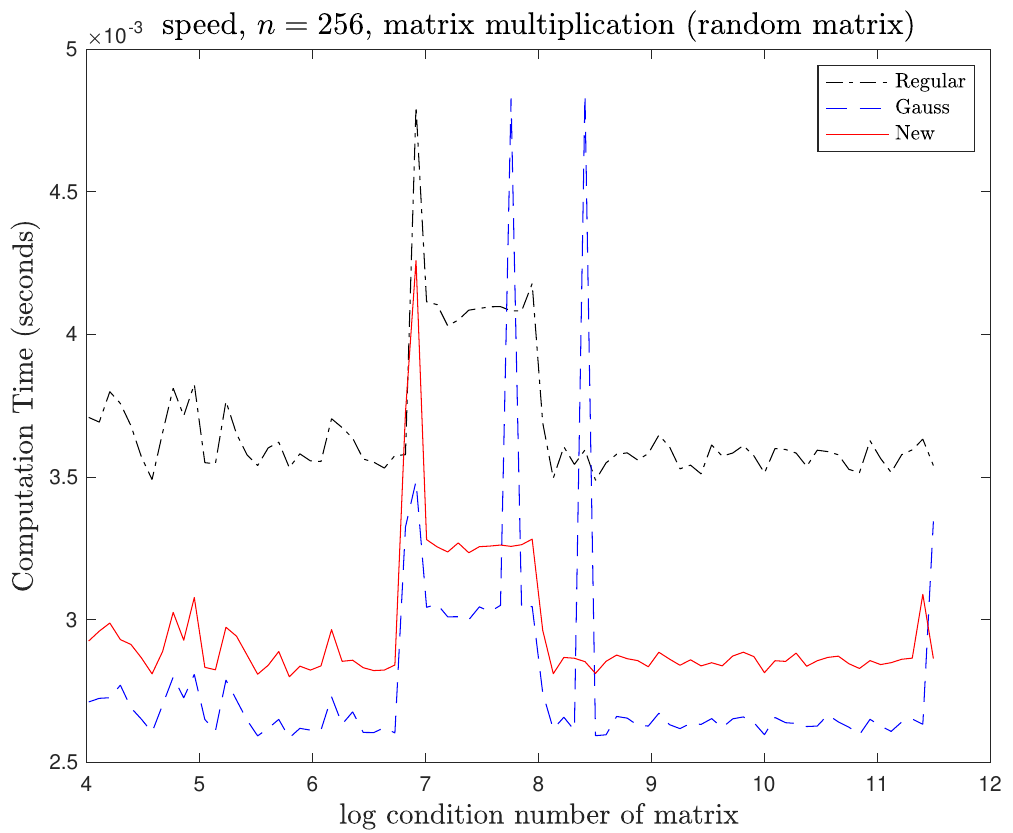}
\caption{Accuracy and speed of algorithms for complex matrix multiplication.}
    \label{fig:Accuracy}
\end{figure}

We generate random $A+i B, C + i D \in \mathbb{C}^{n \times n}$ with $n = 64,128,256$ and with condition numbers ranging from $174$ to $3 \times 10^{11}$.  We use the spectral condition number $\kappa_2(X)$, i.e., ratio of largest to smallest singular values of $X$, throughout this article.  It is desirable to limit ourselves to matrices over Gaussian rationals, i.e., $\mathbb{Q} + \mathbb{Q}i$, as we will need to compute the exact values of their products later.

The way we generate such a matrix requires some elaboration. For an $X \in \mathbb{Z}^{n \times n}$ with a specified $\kappa_2(X) = \kappa \in \mathbb{Z}$. We form a diagonal $\Lambda \in \mathbb{R}^{n \times n}$ whose diagonal entries are $1$ and $\kappa$ toegether with $n-2$ other random integers between $1$ and $\kappa - 1$. We then form $X = H \Lambda H^\tp$ with a random Hadamard matrix $H \in \mathbb{Z}^{n \times n}$. If $A$ and $B$ are generated in this manner, then they are dense matrices (important as we do not want sparsity to unduly influence arithmetic costs) and $\kappa_2(A+iB) = \kappa_2(A) = \kappa_2(B) = \kappa$ as $(\kappa + \kappa i)/(1+i) = \kappa$.

We compute the exact value of $(A+i B)(C + i D)$ symbolically with \textsc{Matlab}'s symbolic toolbox. Given our relatively modest computational resources, this is the bottleneck for our experiments as this step becomes prohibitively expensive when $n > 256$. In generating the $n = 256$ plots in Figure~\ref{fig:Accuracy}, this step alone took 40 hours on our University's Research Computing Center servers.

For each pair of complex matrices $A+i B$ and $C + i D$, we compute their product $\widehat{E}$ using each of the three algorithms \eqref{eq:Regular}, \eqref{eq:Gauss}, \eqref{eq:New}, and compare them against the exact result $E$ via the max norm relative error
\[
    \frac{\lVert E - \widehat{E}\rVert_{\max}}{\lVert A +iB \rVert_{\max} \lVert C + i D \rVert_{\max}}.
\]
As discussed in \cite{Higham,HighamBook}, it is natural to measure error in matrix multiplication relative to the norms of the input matrices. We use the max norm in \eqref{eq:maxnorm} to better capture entrywise accuracy. 

The results are plotted in Figure~\ref{fig:Accuracy}: speed plots have wall time in seconds on the vertical axes; accuracy plots have relative error on the vertical axes; all plots have $\log_{10}(\kappa)$ on the horizontal axes. We repeat each experiment ten times: every value on these plots comes from averaging across the results of ten pairs of random matrices with the same  condition number.

Observations from Figure~\ref{fig:Accuracy}: The accuracy of our new algorithm is much higher than that of Gauss's algorithm and only slightly worse than that of the regular algorithm. Gauss's algorithm also shows a great deal more fluctuation across varying condition numbers than either our new algorithm or the regular one. When it comes to speed, our algorithm is closer to that of Gauss's than the regular algorithm. These accuracy results attest to Theorem~\ref{thm:new} and the discussions around \eqref{eq:compare}.

The relative errors and wall times for \textsc{Matlab}'s internal function for complex matrix multiplication are virtually indistinguishable from those of the regular algorithm (that we implemented ourselves) and thus omitted. In the next three sections, we will compare the accuracy and speed of the three complex matrix multiplication algorithms in more realstic scenarios.

\subsection{Matrix polynomial evaluations}\label{sec:poly}

We evaluate a polynomial $p(x) = \sum_{k = 0}^d a_k x^k$ with coefficients $a_0,\dots,a_k \in \mathbb{R}$ at a $X \in \mathbb{C}^{n \times n}$. This is a problem that occurs in many tasks involving matrix functions \cite{HighamBook,matrix_function}. We limit ourselves to real coefficients as this is by far most common scenario \cite{matrix_function}; but the complex coefficients case simply reduces to evaluating two real polynomials $\Re p(x)$ and $\Im p(x)$. The celebrated Horner's rule \cite[Algorithm~4.3]{matrix_function}, as shown in Algorithm~\ref{alg:poly}, reduces the problem to one of repeated matrix multiplications.
\begin{algorithm}[htb]
\caption{Compute $p(X)$ via Horner's rule}
\label{alg:poly}
\begin{algorithmic}[1]
\renewcommand{\algorithmicrequire}{ \textbf{Input}}
\Require
$a_0,a_1,\dots,a_d \in \mathbb{R}$, $X \in \mathbb{C}^{n \times n}$
\renewcommand{\algorithmicensure}{ \textbf{Output}}
\Ensure
$a_0 I + a_1 X + \dots + a_d X^d$
\State $P = X$;
\State $S = a_0 I + a_1 X$;
\For{$k = 2:d$}
    \State $P = P X$;
    \State $S = S + a_k P$;
\EndFor
\State return $S$;
\end{algorithmic}
\end{algorithm}

We generate random matrices $X \in \mathbb{C}^{256 \times 256}$ with condition numbers from $2^{34}$ to $2^{53}$ as described in Section~\ref{exp:accuracy:general}. We set $d = 5$ and choose random $b_0,\dots,b_5\in (0,1)$ uniformly. We then evaluate $p(X)$ using Algorithm~\ref{alg:poly}, with Step~4 computed via \eqref{eq:Regular}, \eqref{eq:Gauss}, and \eqref{eq:New}.
\begin{figure}[htb]
    \centering
         \includegraphics[width = 0.49\textwidth]{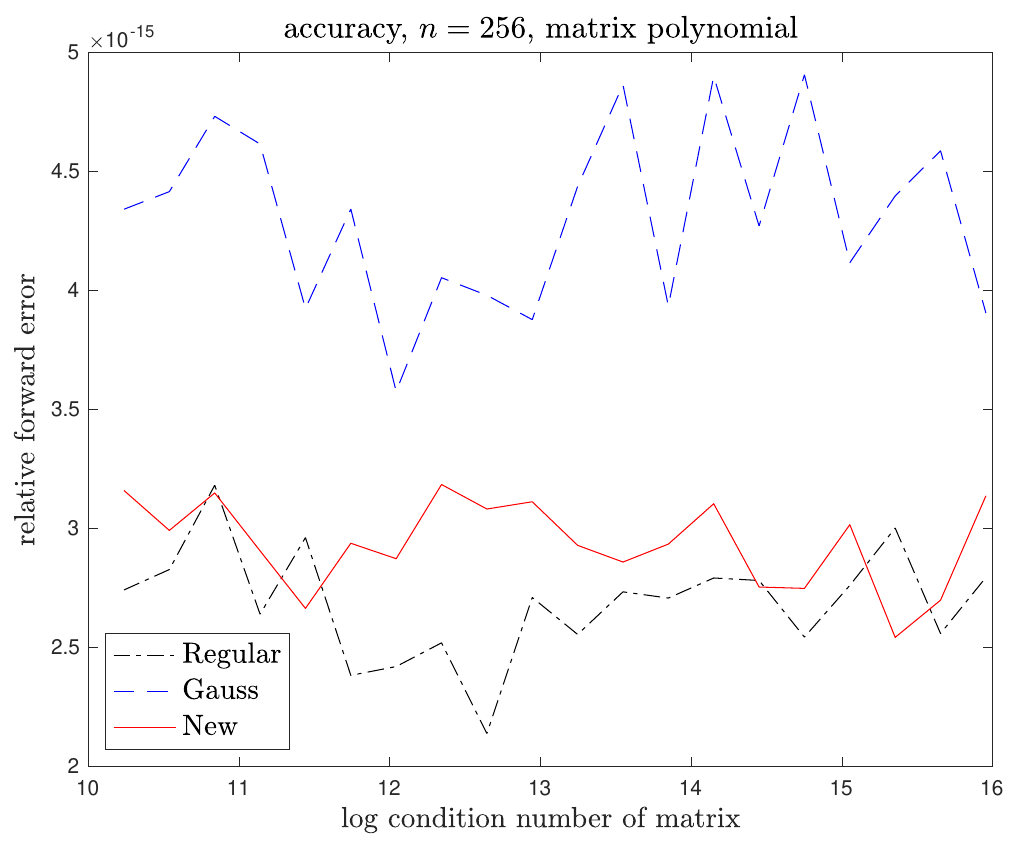}
         \includegraphics[width = 0.49\textwidth]{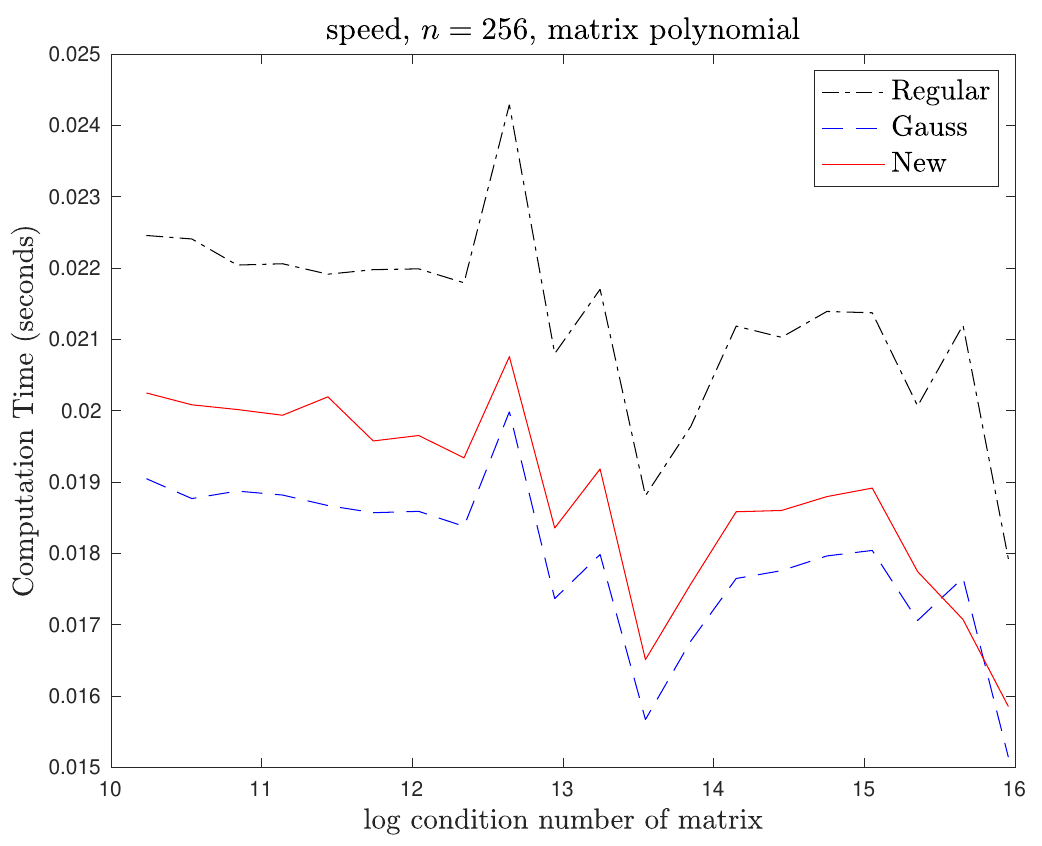}
    \caption{The three algorithms applied to matrix polynomial evaluations.}
    \label{fig:Matrix Polynomial}
\end{figure}
We measure accuracy in terms of the max norm relative forward error
\[
    \frac{\lVert p(X) - \widehat{p}(X)\rVert_{\max}}{\lVert p(X) \rVert_{\max}},
\]
using \textsc{Matlab} symbolic toolbox for the exact value of $p(X)$. The results presented in Figure~\ref{fig:Matrix Polynomial} again show that our new algorithm is nearly as stable as the regular algorithm and nearly as fast as Gauss's algorithm. While our accuracy tests are again limited by our capacity for symbolic computation ($n = 256$ is fine, $n = 512$ is beyond reach), our speed tests can go far beyond (to around $n =4096$), and they show a profile much like Figure~\ref{fig:Speed}.

\subsection{Unitary transforms}\label{sec:uni}

Given a unitary matrix $U \in \mathbb{C}^{n \times n}$ and a complex matrix $X \in \mathbb{C}^{n \times n}$, it may come as a surprise to the reader that unless $U$ happens to be some special transforms like FFT, DCT, DWT, etc, or has already been factored into a product of Householder or Givens matrices,  there is no known special algorithm for forming $UX$ that would take advantage of the unitarity of $U$. Nevertheless, such unitary matrices with no additional special structure are not uncommon. For instance, the matrix $U$ could come from polar decompositions or matrix sign functions \cite{polar_ref2,polar_sylvester_higham,polar_ref3}, and computed via iterative methods \cite{polar_ref2,polar_sylvester_higham,polar_ref3} and thus not in Householder- or Givens-factored form. Here we will explore the use of algorithms \eqref{eq:Regular}, \eqref{eq:Gauss}, \eqref{eq:New} for unitary transforms $X \mapsto UX$.

We generate the unitary matrix $U \in \mathbb{C}^{256 \times 256}$  by QR factoring complex random matrices with entries in $\mathcal{U}[0,1] + \mathcal{U}[0,1]i$. Note that a unitary matrix is always perfectly conditioned. The matrix $X \in \mathbb{C}^{256 \times 256}$ is generated randomly with condition numbers from $2^{34}$ to $2^{53}$ as in Section~\ref{sec:poly}. We compute the exact value $E \coloneqq U X$ symbolically as before and measure the accuracy of our computed value $\widehat{E}$ by
\[
    \frac{\lVert E - \widehat{E}\rVert_{\max}}{\lVert U \rVert_{\max} \lVert X \rVert_{\max}}.
\]
The results, presented in Figure~\ref{fig:Unitary Transformations}, allow us to draw the same conclusion as in the Section~\ref{sec:poly}. Further speed tests up to $n = 4096$ again show a profile much like Figure~\ref{fig:Speed}.

\begin{figure}[ht]
    \centering
         \includegraphics[width = 0.49\textwidth]{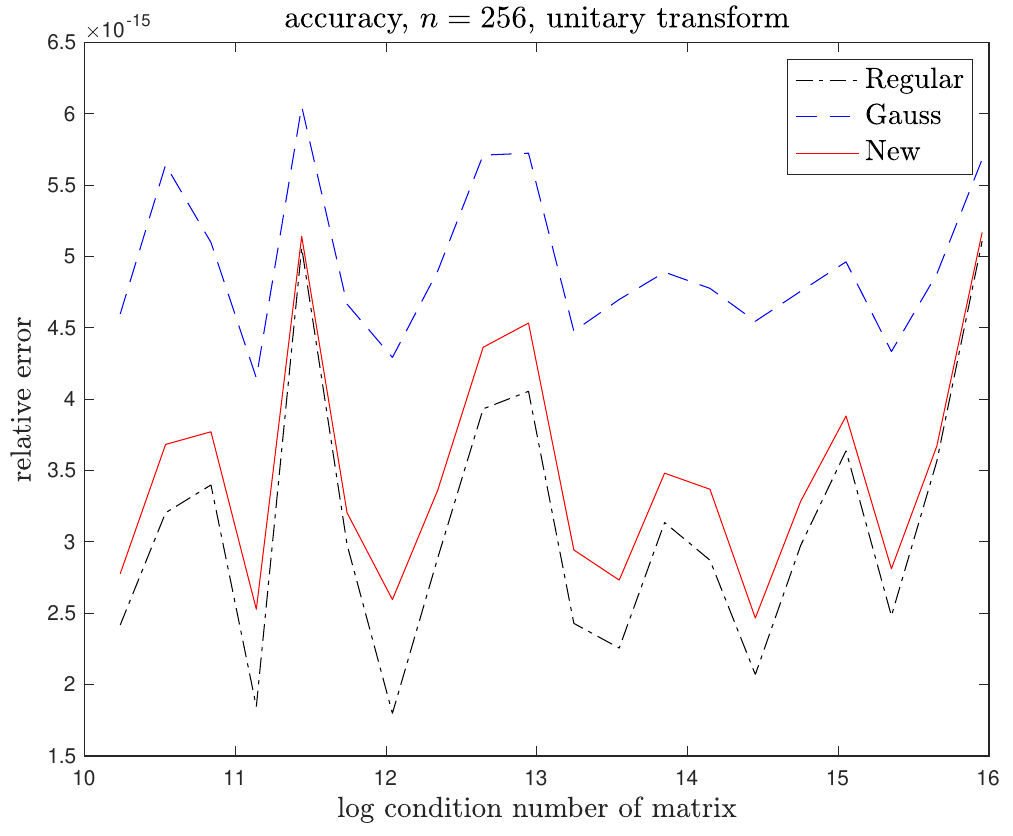}
         \includegraphics[width = 0.49\textwidth]{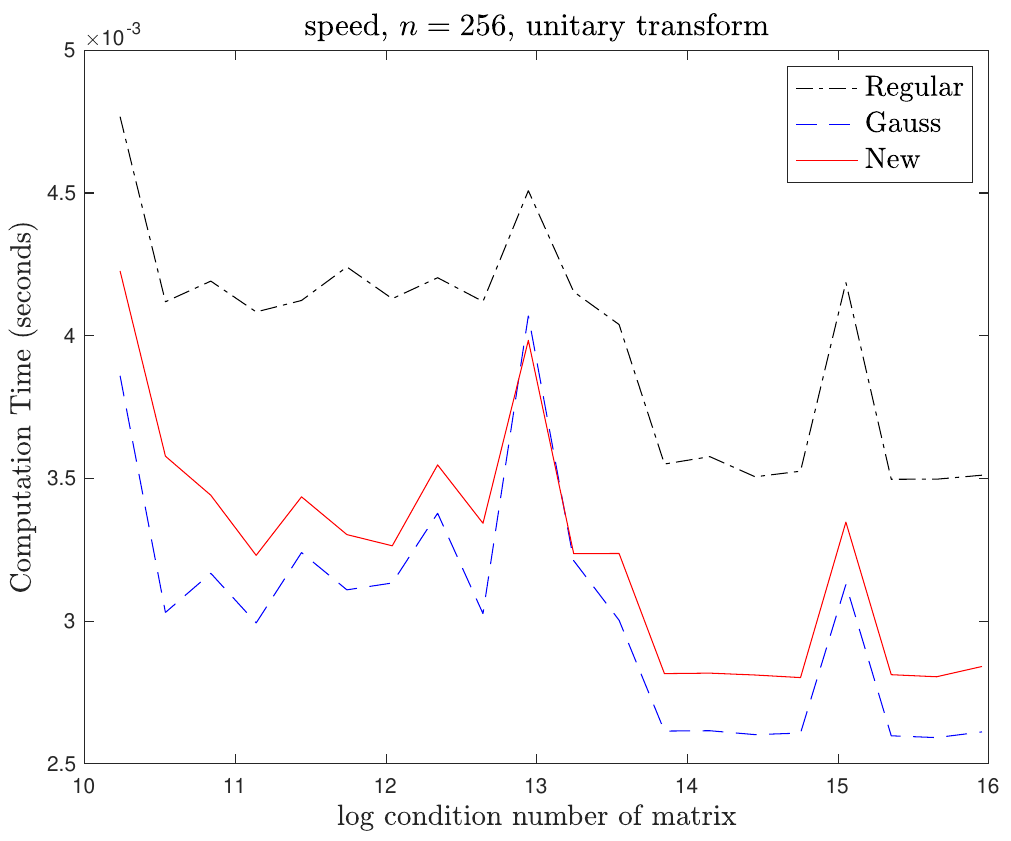}
    \caption{The three algorithms applied to unitary transforms.}
    \label{fig:Unitary Transformations}
\end{figure}

\subsection{Complex-valued neural networks}

A complex-valued neural networks is simply a neural network with complex-valued weights and is activated by a complex function. It has become increasingly important and is widely used in signal processing and computer vision \cite{complexNN_1,complexNN_6,complexNN_3,complexNN_4,complexNN_5,complexNN_2}. For simplicity, we consider a $d$-layer constant width version $f :\mathbb{C}^n \to \mathbb{C}^n$ given by
\[
f(W_1, \dots, W_d, \sigma)(x) \coloneqq W_d \sigma (W_{d-1} \sigma( \cdots W_2 \sigma(W_1 x) \cdots )),
\]
with weight matrices $W_1, \dots, W_d \in \mathbb{C}^{n \times n}$ and activation function $\sigma : \mathbb{C} \to \mathbb{C}$ applied coordinatewise on $\mathbb{C}^n$, as depicted in Figure~\ref{fig:neural network}.
\begin{figure}[ht]
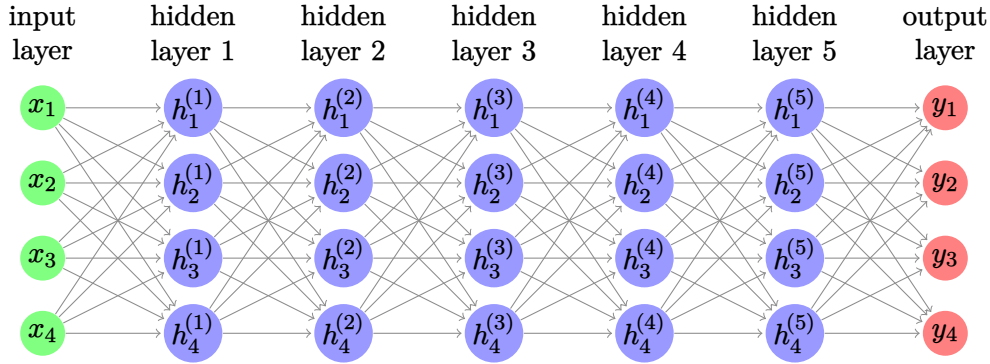

    \centering
    \begin{neuralnetwork}[height=4,layerspacing=20mm]
        \newcommand{\x}[2]{$x_#2$}
        \newcommand{\y}[2]{$y_#2$}
        \newcommand{\hfirst}[2]{ $h^{\!(1)}_#2$}
        \newcommand{\hsecond}[2]{ $h^{\!(2)}_#2$}
        \newcommand{\hthird}[2]{ $h^{\!(3)}_#2$}
        \newcommand{\hforth}[2]{ $h^{\!(4)}_#2$}
        \newcommand{\hfifth}[2]{ $h^{\!(5)}_#2$}
        \inputlayer[count=4, bias=false, title=input\\layer, text=\x]
        \hiddenlayer[count=4, bias=false, title=hidden\\layer 1, text=\hfirst] \linklayers
        \hiddenlayer[count=4, bias=false, title=hidden\\layer 2, text=\hsecond] \linklayers
        \hiddenlayer[count=4, bias=false, title=hidden\\layer 3, text=\hthird] \linklayers
        \hiddenlayer[count=4, bias=false, title=hidden\\layer 4, text=\hforth] \linklayers
        \hiddenlayer[count=4, bias=false, title=hidden\\layer 5, text=\hfifth] \linklayers
        \outputlayer[count=4, title=output\\layer, text=\y] \linklayers
    \end{neuralnetwork}
    \caption{A constant width neural network with input dimension $n = 4$ and depth $d = 6$. The arrows between adjacent layers are weighted with values in the weight matrices. $h^{(k)} \in \mathbb{R}^n$ denotes the output of the $k$th layer.}
\label{fig:neural network}
\end{figure}

Complex matrix multiplications are indispensable when we \emph{train} (i.e., fit with data in order to determine the weights $W_1,\dots,W_d$) such a neural network through backpropagation, or when we \emph{evaluate} it on multiple inputs $x_1,\dots, x_m \in \mathbb{C}^n$ to make new predictions. Here we will compare the performance of the three algorithms  \eqref{eq:Regular}, \eqref{eq:Gauss}, \eqref{eq:New} for the latter task as it allows for easier control of the condition numbers of $W_1,\dots,W_d$.

For concreteness, we choose a depth of $d = 6$ and use the complex ReLU activation \cite{complexNN_6,complexNN_5}
\[
    \sigma(a+b i) \coloneqq \max(a,0) + \max(b,0) i.
\]
We generate random weight matrices $W_1, \dots, W_6 \in \mathbb{C}^{n \times n}$ with $n = 64$ and $128$, and with condition numbers ranging from $2^{34}$ to $2^{53}$. We also generate random inputs $X = [x_1,\dots,x_m] \in \mathbb{C}^{n \times m}$ with entries drawn from $\mathcal{U}[-\frac12,\frac12] +\mathcal{U}[-\frac12,\frac12]i$, and with $(m,n) = (25,64)$ or $(50,128)$. The task is then to evaluate
\[
    E \coloneqq f(W_1, \dots, W_d, \sigma)(X) \coloneqq W_d \sigma (W_{d-1} \sigma( \cdots W_2 \sigma(W_1 X) \cdots )).
\]
Again we compute its exact value $E$ symbolically, apply the three algorithms to obtain $\widehat{E}$ numerically, and measure accuracy in terms of the relative forward error
\[
    \frac{\lVert E-\widehat{E}\rVert_{\max}}{\lVert E \rVert_{\max}}.
\]
The results, shown in Figure~\ref{fig:Complex Neural Networks}, are fully consistent with those in Sections~\ref{sec:poly} and \ref{sec:uni}.

\begin{figure}[ht]
    \centering
         \includegraphics[width = 0.49\textwidth]{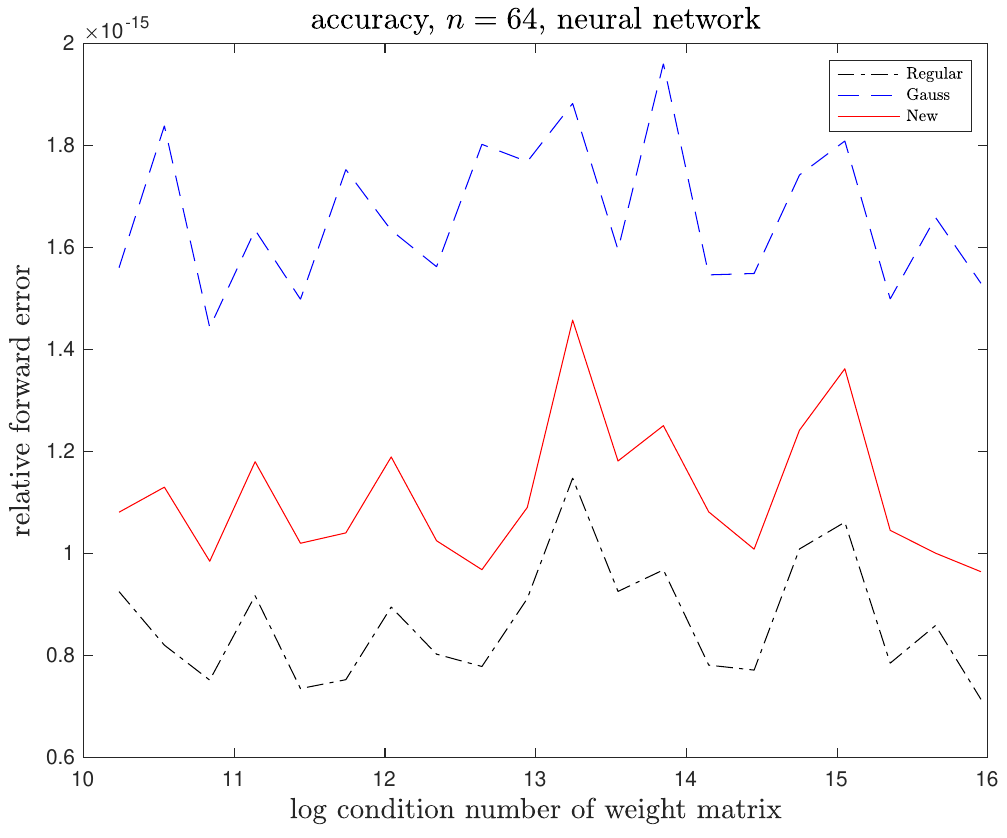}
         \includegraphics[width = 0.49\textwidth]{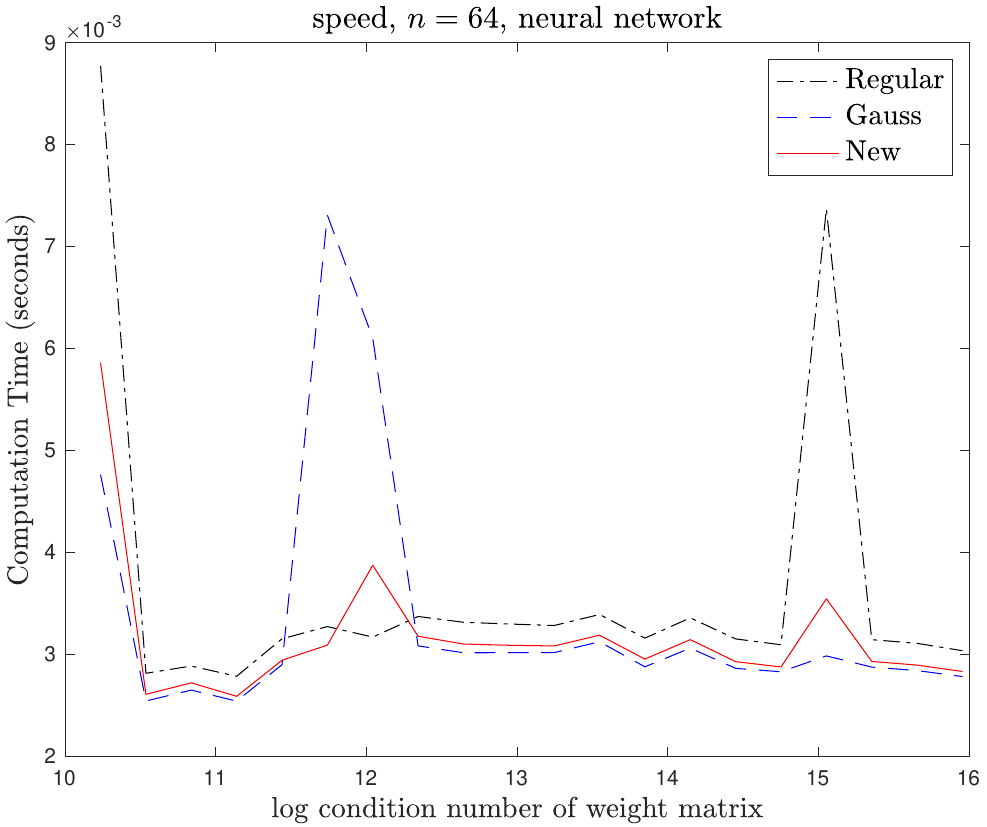}
         \includegraphics[width = 0.49\textwidth]{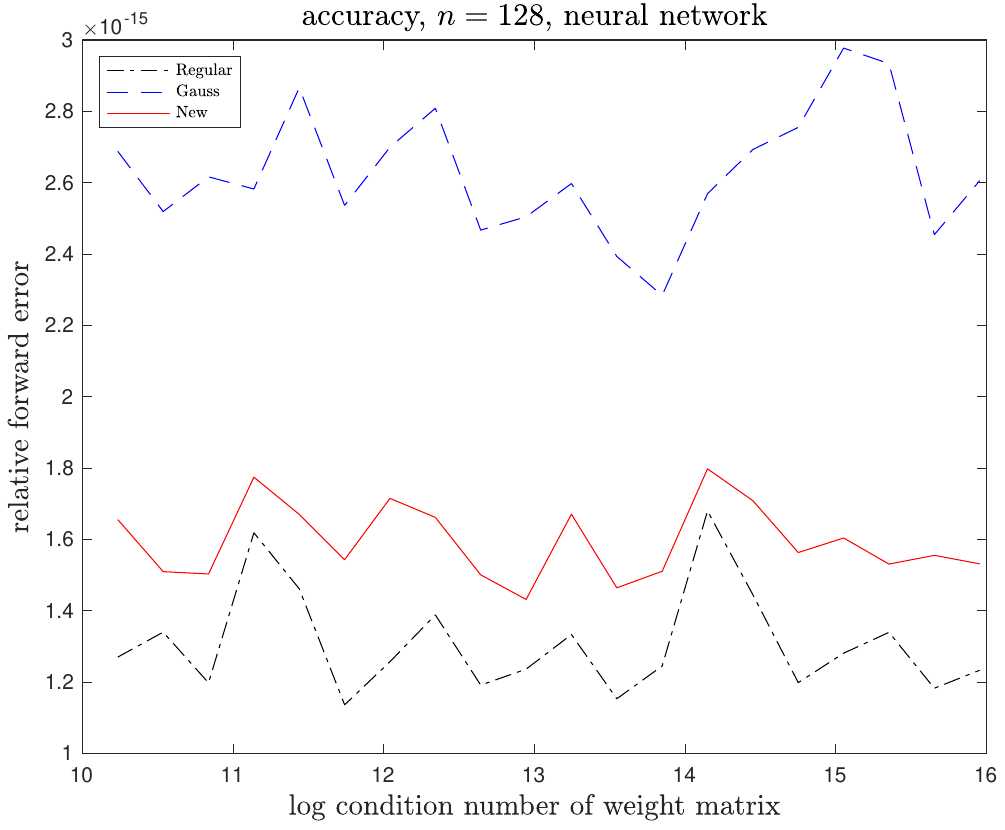}
         \includegraphics[width = 0.49\textwidth]{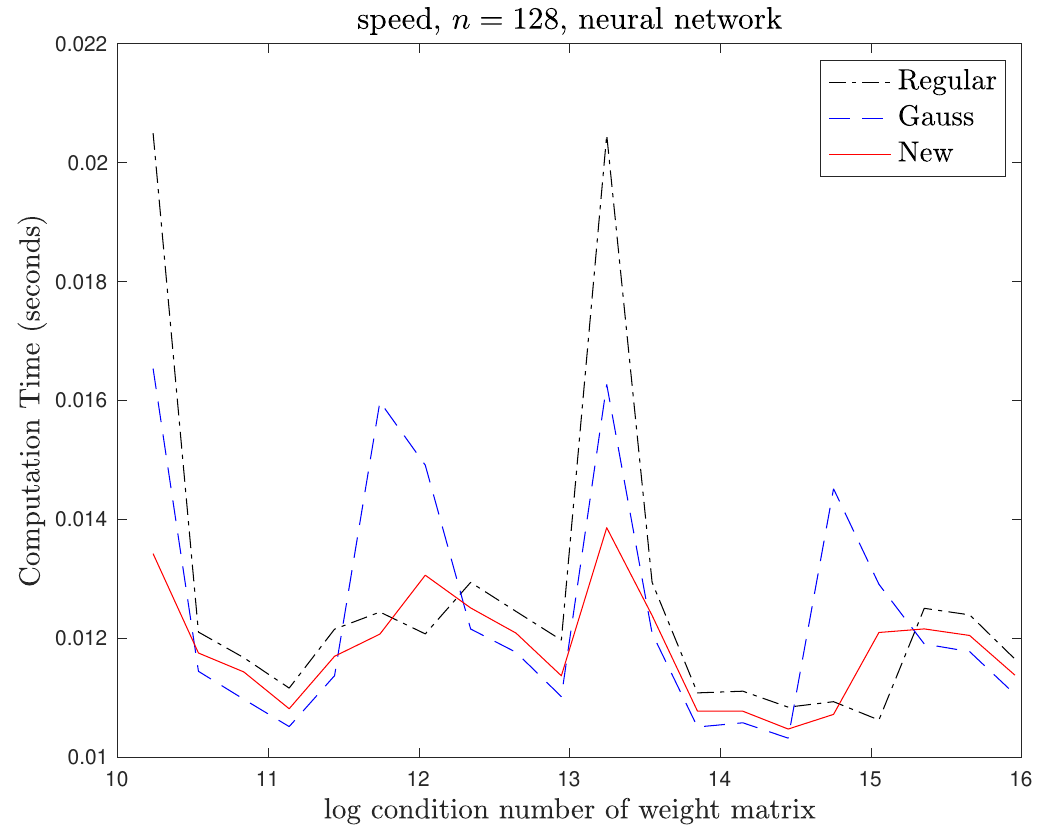}
    \caption{The three algorithms applied to $6$-layer complex neural networks with complex ReLU activation and widths $64$ and $128$.}
    \label{fig:Complex Neural Networks}
\end{figure}

\

\section{Conclusion}

The notion of bilinear complexity started by Strassen has been a great motivator for more than five decades of exciting developments in numerical linear algebra. Its success illustrates the adage that ``less is more''. Bilinear complexity does not capture every operation that underlies the speed of an algorithm; but by focusing on a single operation (variable multiplications) and disregarding the rest (e.g., scalar multiplications, additions), it allows speed to be measured by the number of terms in a decomposition of a $3$-tensor and the fastest algorithm to be given by a rank decomposition. This opens a door to other areas of mathematics like algebraic geometry where such decompositions are studied independent of their computational relevance.

We hope the notion of bilinear stability proposed in this article would do for the study of numerical stability what bilinear complexity did for the study of time complexity. By focusing on a single factor (growth) and disregarding other factors (e.g., cancellation errors) that play a role in numerical stability, it allows stability to be measured by the growth factor in a decomposition of a $3$-tensor and the stablest algorithm to be given by a nuclear decomposition. Just as tensor rank connects to algebraic geometry, tensor nuclear norm connects to functional analysis \cite{Defant, Diestel, Ryan}; thus bilinear stability could potentially open a door to this rich area of mathematics.

A very recent development in bilinear complexity is the automated discovery of fast algorithms using deep reinforcement learning. In \cite{DeepMind}, \emph{AlphaTensor} found more than 14,000 inequivalent $49$-term decompositions for $4 \times 4$ matrix product. This is impressive. But when one has that many different algorithms the question becomes which one to pick? From the perspective of numerical linear algebra, numerical stability would be the most natural secondary criteria. Since the 14,000 algorithms are all given in the form of $49$-term decompositions, their growth factors are trivial to calculate and all one needs to do is to pick the decomposition with the smallest growth factor.

\subsection*{Acknowledgment} The authors would like to thank Nick Higham and Ke Ye for helpful discussions, the two anonymous reviewers for their very pertinent suggestions, and the University of Chicago's Research Computing Center for its computing resources and services. ZD acknowledges the support of DARPA HR00112190040 and NSF ECCF 2216912. LHL acknowledges the support of DARPA HR00112190040, NSF DMS 1854831, and a Vannevar Bush Faculty Fellowship ONR  N000142312863.

\bibliographystyle{abbrv}
\bibliography{stability-ref} 

\end{document}